\documentclass[12pt, reqno]{amsart}
\usepackage{amssymb,amsthm,amsfonts,amsmath}

\usepackage{hyperref}
\usepackage{mathrsfs}

\usepackage[dvips]{graphicx}

\newcommand\Y{\mathbb Y}
\newcommand\Z{\mathbb Z}
\newcommand\C{\mathbb C}
\newcommand\R{\mathbb R}

\newcommand\al{\alpha}

\newcommand\Ga{\Gamma}
\newcommand\de{\delta}

\newcommand\la{\lambda}
\newcommand\si{\sigma}
\newcommand\tsi{\widetilde\si}
\newcommand\epsi{\varepsilon}
\newcommand\om{\omega}

\newcommand\Conf{\operatorname{Conf}}
\newcommand\const{\operatorname{const}}
\newcommand\tr{\operatorname{tr}}

\newcommand\inv{\operatorname{inv}}
\newcommand\out{{\operatorname{\,out}}}
\newcommand\inn{{\operatorname{\,in}}}
\newcommand\sparse{{\text{\rm sparse}}}

\newcommand\wt{\widetilde}

\newcommand\LL{\mathcal L_{1|2}}

\newcommand\unP{\underline P\,}
\newcommand\unK{\underline K\,}

\newcommand\unX{\underline X}

\newcommand\X{\mathfrak X}

\newtheorem{theorem}{Theorem}[section]
\newtheorem{proposition}[theorem] {Proposition}

\newtheorem{corollary}[theorem]{Corollary}
\newtheorem{lemma}[theorem]{ Lemma}

\theoremstyle{definition}
\newtheorem{definition}[theorem]{Definition}
\newtheorem{remark}[theorem]{Remark}
\newtheorem{example}[theorem]{Example}

\numberwithin{equation}{section}

\newtheorem{claim}{Claim}
\newtheorem{problem}{Problem}

\begin{document}

\title[The quasi-invariance property]
{The quasi-invariance property for the Gamma kernel determinantal measure}

\author{Grigori Olshanski}
\address{Institute for Information Transmission Problems, Bolshoy
Karetny 19,  Moscow 127994, Russia; Independent University of Moscow, Russia}
\email{olsh2007@gmail.com}

\date{November 19, 2009}

\thanks{At various stages of work, the present
research was supported by the RFBR grants 07-01-91209 and 08-01-00110, by the
project SFB 701 (Bielefeld University), and by the grant ``Combinatorial
Stochastic Processes" (Utrecht University).}

\begin{abstract}
The Gamma kernel is a projection kernel of the form
$(A(x)B(y)-B(x)A(y))/(x-y)$, where $A$ and $B$ are certain functions on the
one-dimensional lattice expressed through Euler's $\Ga$-function. The Gamma
kernel depends on two continuous parameters; its principal minors  serve as the
correlation functions of a determinantal probability measure $P$ defined on the
space of infinite point configurations on the lattice. As was shown earlier
(Borodin and Olshanski, Advances in Math. 194 (2005), 141-202;
arXiv:math-ph/0305043), $P$ describes the asymptotics of certain ensembles of
random partitions in a limit regime.

Theorem: The determinantal measure $P$ is quasi-invariant with respect to
finitary permutations of the nodes of the lattice.

This result is motivated by an application to a model of infinite particle
stochastic dynamics.
\end{abstract}

\maketitle

\tableofcontents

\section*{Introduction}

\subsection{Preliminaries: a general problem}\label{0-1}
Recall a few well-known notions from measure theory. Let $\mathfrak A$ be a
Borel space (that is, a set with a distinguished sigma-algebra of subsets). Two
Borel measures $P_1, P_2$ on $\mathfrak A$ are said to be {\it equivalent\/} if
$P_1$ has a density with respect to $P_2$ and vice versa. They are said to {\it
disjoint\/} or {\it mutually singular\/} if there exist disjoint Borel subsets
$B_1$ and $B_2$ such that $P_1$ is supported by $B_1$ and $P_2$ is supported by
$B_2$ (that is, $P_1(\mathfrak A\setminus B_1)=P_2(\mathfrak A\setminus
B_2)=0$). Assume $G$ is a group acting on $\mathfrak A$ by Borel
transformations; then a Borel measure $P$ is said to be $G$-{\it
quasi-invariant\/} if $P$ is equivalent to its transform by any element $g\in
G$.

In practice, especially for measures living on ``large'' spaces, verifying the
property of equivalence, disjointness or quasi-invariance, and explicit
computation of densities (Radon--Nikod\'ym derivatives) for equivalent measures
can be a nontrivial task. There exist nice general results for particular
classes of measures: infinite product measures (Kakutani's theorem \cite{Ka}),
Gaussian measures on infinite-dimensional spaces (Feldman--Hajek's theorem and
related results, see \cite[Ch. II]{Kuo}), Poisson measures (see \cite{Bro}).

Assume that $\X$ is a locally compact space, take as the ``large'' space
$\mathfrak A$ the space $\Conf(\X)$ of locally finite point configurations on
$\X$, and assume that the measures under consideration are probability measures
on $\mathfrak A=\Conf(\X)$; they are also called {\it point processes\/} on
$\X$ (for fundamentals of point processes, see, e.g., \cite{Le}). Poisson
measures are just the simplest yet important example of point processes. The
next by complexity example is the class of {\it determinantal\/} measures
(processes). Determinantal measures are specified by their correlation kernels
which are functions $K(x,y)$ on $\X\times\X$. Note an analogy with covariation
kernels of Gaussian measures which are also functions in two variables. Note
also that, informally, Poisson measures can be viewed as a degenerate case of
determinantal measures corresponding to kernels $K(x,y)$ concentrated on the
diagonal $x=y$.

Many concrete examples of determinantal measures are furnished by random matrix
theory and other sources, see, e.g., the surveys \cite{So} and \cite{Bor}. The
interest to determinantal measures especially increased in the last years.
However, to the best of my knowledge, the following problem was never discussed
in the literature:

\begin{problem}\label{P1}
{\it Assume we are given two determinantal measures, $P_1$ and $P_2$ on a
common space $\Conf(\X)$. How to test their equivalence (or, on the contrary,
disjointness)? Is it possible to decide this by inspection of the respective
correlation kernels $K_1(x,y)$ and $K_2(x,y)$?}
\end{problem}

One could imagine that equivalence $P_1\sim P_2$ holds if the kernels are close
to each other in an appropriate sense. However, there is a subtlety here, see
Subsection \ref{1-6} below.

Let $G$ be a group of homeomorphisms $g:\X\to\X$. Then $G$ also acts, in a
natural way, on the space $\Conf(\X)$ and hence on the space of probability
measures on $\Conf(\X)$. Observe that the latter action preserves the
determinantal property: If $P$ is a determinantal measure on $\Conf(\X)$ with
correlation kernel $K(x,y)$, then the transformed measure $g(P)$ is
determinantal, too, and $K(g^{-1}x, g^{-1}y)$ serves as its correlation kernel.
Thus, the question of $G$-quasi-invariance of $P$ becomes a special instance of
Problem \ref{P1}:

\begin{problem}\label{P2}
{\it Let $P$ and $G$  be as above. How to test whether $P$ is
$G$-quasi-invariant? Is it possible to decide this by comparing the correlation
kernels $K(x,y)$ and $K(g^{-1}x, g^{-1}y)$ for $g\in G$?}
\end{problem}

I think it would be interesting to develop general methods for solving Problem
\ref{P2} and more general Problem \ref{P1}. They seem to be nontrivial even in
the case when $\X$ is a countably infinite set with discrete topology.

\subsection{The Gamma kernel measure}
In the present paper we are dealing with a concrete model of determinantal
measures, introduced in \cite{BO2}. The space $\X$ is assumed to be discrete
and countable; it is convenient to identify it with the lattice
$\Z':=\Z+\frac12$ of half--integers. Then the space $\Conf(\X)=\Conf(\Z')$ is
simply the space of all subsets of $\Z'$. We consider a two-parameter family of
kernels on $\Z'\times\Z'$. Following \cite{BO2}, we denote them as
$\unK_{z,z'}(x,y)$; here $z$ and $z'$ are some continuous parameters, and $x,y$
are the arguments, which range over $\Z'$. Each kernel is real-valued and
symmetric. Moreover, it is a projection kernel meaning that it corresponds to a
projection operator in the Hilbert space $\ell^2(\Z')$. Like many examples of
kernels from random matrix theory, our kernels can be written in the so-called
integrable form \cite{IIKS}, \cite{De}
$$
\frac{\mathcal A(x)\mathcal B(y)-\mathcal B(x)\mathcal A(y)}{x-y},
$$
resembling Christoffel--Darboux kernels associated to orthogonal polynomials.
In our situation $\mathcal A$ and $\mathcal B$ are certain functions on the
lattice $\Z'$, which are expressed through Euler's $\Ga$-function. For this
reason we call $\unK_{z,z'}(x,y)$ the {\it Gamma kernel\/}. In \cite{BO2} we
conjectured that the Gamma kernel might be a universal microscopic limit of the
Christoffel--Darboux kernels for generic discrete orthogonal polynomials, in an
appropriate asymptotic regime.

The Gamma kernel serves as the correlation kernel for a determinantal measure
on $\Conf(\Z')$, called the {\it Gamma kernel measure\/} and denoted as
$\unP_{z,z'}$. According to the general definition of determinantal measures
(see \cite{So}, \cite{Bor}), the measure $\unP_{z,z'}$ is characterized by its
correlation functions
$$
\rho_n(x_1,\dots,x_n):=\unP_{z,z'}\{X\in\Conf(\Z')\mid X\ni x_1,\dots,x_n\}
$$
which in turn are equal to principal $n\times n$ minors of the kernel:
$$
\rho(x_1,\dots,x_n):=\det[\unK_{z,z'}(x_i,x_j)]_{i,j=1}^n.
$$
Here $n=1,2,\dots$ and $x_1,\dots,x_n$ is an arbitrary  $n$-tuple
$x_1,\dots,x_n$ of pairwise distinct points from $\Z'$.

As shown in \cite{BO2}, the Gamma kernel measure arises from several models of
representation--theoretic origin, through certain limit transitions.

A more detailed information about $\unK_{z,z'}(x,y)$ and $\unP_{z,z'}$ is given
in Section \ref1 below, see also \cite{BO2}, \cite{Ol2}.

\subsection{The main result}
We take as $G$ the group $\mathfrak S$ of permutations of the set $\Z'$ fixing
all but finitely many points. Such permutations are said to be {\it
finitary\/}. Clearly, $\mathfrak S$ is a countable group. It is generated by
the elementary transpositions $\si_n$ of the lattice $\Z'$: Here $n\in\Z$ and
$\si_n$ transposes the points $n-\frac12$ and $n+\frac12$ of $\Z'$. Each
permutation $\si\in\mathfrak S$ induces, in a natural fashion, a transformation
of the space $\Conf(\Z')$, which in turn results in a transformation $P\mapsto
\si(P)$ of probability measures on $\Conf(\Z')$.

The main result of the paper says that the Gamma kernel measure is
quasi-invariant with respect to the action of the group $\mathfrak S$:

\medskip
\noindent{\bf Main Theorem.} {\it For any $\sigma\in\mathfrak S$, the measures
$\unP_{z,z'}$ and $\si(\unP_{z,z'})$ are equivalent. Moreover, the
Radon--Nikod\'ym derivative $\si(\unP_{z,z'})/\unP_{z,z'}$ can be explicitly
computed}.
\medskip

This result gives a solution to the first question of Problem \ref{P2} in a
concrete situation. As will be shown in another paper, the quasi-invariance
property established in the theorem  makes it possible to construct an
equilibrium Markov process on $\Conf(\Z')$ with determinantal dynamical
correlation functions and equilibrium distribution $\unP_{z,z'}$. This
application is one of the motivations of the present work. \footnote{A
connection between quasi-invariance and existence of  Markov dynamics,
sometimes in hidden form, is present in various situations. See, e.g.,
\cite{AKR}, \cite{SY}.}

It seems plausible that $\unP_{z,z'}$ is {\it not\/} quasi-invariant with
respect to the transformations of $\Conf(\Z')$ generated by the translations of
the lattice. Note that the translation $x\mapsto x+k$ with $k\in\Z$ amounts to
the shift $(z,z')\to(z-k,z'-k)$ of the parameters (see Theorem \ref{1.4}). One
can ask, more generally, whether any two Gamma kernel measures with distinct
parameters are disjoint. \footnote{The pair $(z,z')$ should be viewed as an
{\it unordered\/} pair of parameters, because the transposition
$z\leftrightarrow z'$ does not affect the measure, see Theorem \ref{1.4}.}

\subsection{Scheme of proof of Main Theorem}\label{0-4} The proof relies on
the fact that for fixed $z,z'$, the measure $\unP_{z,z'}$ can be approximated
by simpler measures which are $\mathfrak S$--quasiinvariant and whose
Radon--Nikod\'ym derivatives (with respect to the action of the group
$\mathfrak S$) are readily computable.

The approximating measures depend on an additional parameter $\xi\in(0,1)$ and
are denoted as $\unP_{z,z',\xi}$. These are purely atomic probability measures
supported by a single $\mathfrak S$--orbit. They come from certain probability
distributions on Young diagrams, and are called the {\it z--measures\/}
(Kerov--Olshanski--Vershik \cite{KOV}, Borodin--Olshanski \cite{BO1}). As $\xi$
goes to $1$, the measures $\unP_{z,z',\xi}$ weakly converge to $\unP_{z,z'}$:
this is simply the initial definition of $\unP_{z,z'}$ given in \cite{BO2}.

What we actually need to prove is that the convergence of the measures holds
not only in the weak topology (that is, on bounded continuous test functions)
but also in a much stronger sense: Namely,
$$
\langle F, \unP_{z,z',\xi}\rangle\, \underset{\xi\to1}{\longrightarrow}\,
\langle F, \unP_{z,z'}\rangle
$$
for certain test functions $F$ which, like the Radon--Nikod\'ym derivatives,
may be unbounded and not everywhere defined. Here and in the sequel the angular
brackets denote the pairing between functions and measures.

To explain this point more precisely we need some preparation.

First of all, it is convenient to transform all the measures in question by
means of an involutive homeomorphism of the compact space $\Conf(\mathfrak
\Z')$. This homeomorphism, denoted as ``$\inv$'', assigns to a configuration
$X\in\Conf(\Z')$ its symmetric difference with the set
$\Z'_-=\{\dots,-\frac32,-\frac12\}$.

An equivalent description is the following. Regard $X$ as a configuration of
charged particles occupying some of the sites of the lattice $\Z'$, while the
holes (that is, the unoccupied sites of $\Z'$) are interpreted as
anti--particles with opposite charge.  Now, the new configuration $\inv(X)$ is
formed by the particles sitting to the right of 0 and the anti--particles to
the left of 0. We call ``$\inv$'' the {\it particle/hole involution\/} on
$\Z'_-$.

For instance, if $X=\Z'_-$ then $\inv(X)=\varnothing$, the empty configuration.
The configuration $X=\Z'_-$ plays a distinguished role because the $\mathfrak
S$--orbit of this configuration is the support of the pre--limit measures
$\unP_{z,z',\xi}$. The map ``$\inv$'' transforms this distinguished orbit into
the set of all finite {\it balanced\/} configurations, that is, finite
configurations with equally many points to the right and to the left of $0$.

Note that the transform by ``$\inv$'' leaves intact the action of all the
elementary transpositions $\si_n$ with $n\ne0$, only the action of $\si_0$ is
perturbed.

Note also that if $\unP$ is a determinantal measure on $\Conf(\Z')$ then so is
its push--forward $P:=\inv(\unP)$, and there is a simple relation between the
correlation kernels of $\unP$ and $P$ (\cite[Appendix]{BOO}). If the kernel of
$\unP$ is symmetric then that of $P$ has a different kind of symmetry: it is
symmetric with respect to an indefinite inner product (see \cite[Proposition
2.3 and Remark 2.4]{BO1}), which reflects the presence of two kinds of
particles.

Instead of the measures $\unP_{z,z',\xi}$ and $\unP_{z,z'}$ we will deal with
their transforms by ``$\inv$'', denoted as
$P_{z,z',\xi}:=\inv(\unP_{z,z',\xi})$ and $P_{z,z'}:=\inv(\unP_{z,z'})$.
Clearly, the transform does not affect the formulation of the theorem, only the
initial action of the group $\mathfrak S$ on $\Conf(\Z')$ has to be conjugated
by the involution: an element $\si\in\mathfrak S$ now acts as the
transformation
\begin{equation}\label{f0.1}
\tsi:=\inv\circ\,\si\circ\inv.
\end{equation}

An advantage of the transformed measures as compared to the initial ones is
that the pre--limit measures $P_{z,z',\xi}$ live on finite configurations. In a
weaker form, this property is inherited by the limit measures. Namely, let us
say that a configuration $X\in\Conf(\Z')$ is {\it sparse\/} if
$$
\sum_{x\in X}|x|^{-1}<\infty.
$$
Denote the set of all sparse configurations as $\Conf_\sparse(\Z')$.  There is
a natural embedding $\Conf_\sparse(\Z')\hookrightarrow\ell^1(\Z')$ assigning to
a sparse configuration $X$ its characteristic function multiplied by the
function $|x|^{-1}$, and we equip $\Conf_\sparse(\Z')$ with the
``$\ell^1$--topology'', that is, the one induced by the norm of the Banach
space $\ell^1(\Z')$. The $\ell^1$--topology is finer than the topology induced
from the ambient space $\Conf(\Z')$.

Now we are in a position to describe the scheme of proof.

\begin{claim}\label{C1}
{\it The limit measures $P_{z,z'}$ are concentrated on the set of sparse
configurations.}
\end{claim}

The claim makes sense because the set $\Conf_\sparse(\Z')$ is a Borel subset in
$\Conf(\Z')$.

Given a function $f$ on the lattice $\Z'$ such that $|f(x)|=O(|x|^{-1})$, we
define a function $\Phi_f(X)$ on the set of sparse configurations by the
formula
\begin{equation}\label{f0.mult}
\Phi_f(X)=\prod_{x\in X}(1+f(x))
\end{equation}
(the product is convergent). Such functions $\Phi_f$ will be called {\it
multiplicative functionals\/} on configurations.  Any multiplicative functional
$\Phi_f$ is continuous in the $\ell^1$--topology.

Given a permutation $\si\in\mathfrak S$ and a measure $P$ on $\Conf(\Z')$, we
denote by $\tsi(P)$ the push--forward of $P$ under the transformation $\tsi$,
see \eqref{f0.1}.

Let $z,z'$ be fixed and $\xi$ range over $(0,1)$. For any $\si\in\mathfrak S$,
let $\mu_{z,z',\xi}(\si,X)$ be the Radon--Nikod\'ym derivative of the measure
$\tsi(P_{z,z',\xi})$ with respect to the measure $P_{z,z',\xi}$. That is,
$$
\mu_{z,z',\xi}(\si,X)=\frac{\tsi(P_{z,z',\xi})(X)}{P_{z,z',\xi}(X)}
=\frac{P_{z,z',\xi}(\tsi^{-1}(X))}{P_{z,z',\xi}(X)};
$$
here $X$ belongs to the countable set of finite balanced configurations.

\begin{claim}\label{C2}
{\it Fix an arbitrary $\si\in\mathfrak S$.

{\rm(i)} The function $\mu_{z,z',\xi}(\si,X)$ has a unique extension to a
continuous function on $\Conf_\sparse(\Z')$.

{\rm(ii)} As $\xi\to1$, the extended functions obtained in this way converge
pointwise to a continuous function $\mu_{z,z'}(\si,X)$ on $\Conf_\sparse(\Z')$.

{\rm(iii)} The limit function $\mu_{z,z'}(\si,X)$ can be written as a finite
linear combination of multiplicative functionals of the form \eqref{f0.mult}.}
\end{claim}

Here continuity is assumed with respect to the $\ell^1$--topology. Actually, a
somewhat stronger claim holds, see Proposition \ref{3.1} and the subsequent
discussion.

Claim \ref{C2} suggests that the limit function $\mu_{z,z'}(\si,X)$ might serve
as the Radon--Nikod\'ym derivative for the limit measure, that is,
$$
\tsi(P_{z,z'})=\mu_{z,z'}(\si,\,\cdot\,)P_{z,z'}.
$$

This relation is indeed true. We reduce it to the following claim.

\begin{claim}\label{C3} {\it Let $f$ be an arbitrary  function on $\Z'$ such that
$|f(x)|=O(|x|^{-1})$. Then the multiplicative functional $\Phi_f$ given by
\eqref{f0.mult} is absolutely integrable with respect to both the pre--limit
and limit measures, and we have}
$$
\lim_{\xi\to1}\langle\Phi_f,P_{z,z',\xi}\rangle=\langle\Phi_f,P_{z,z'}\rangle.
$$
\end{claim}

This claim is stronger than the assertion about the weak convergence of
measures $P_{z,z',\xi}\to P_{z,z'}$ which was known previously. Indeed, weak
convergence of measures on $\Conf(\Z')$ means convergence on continuous test
functions, while multiplicative functionals are, generally speaking, unbounded
functions on $\Conf_\sparse(\Z')$ and thus cannot be extended to continuous
functions on the compact space $\Conf(\Z')$.

To prove Claim \ref{C3} we use the well--known fact that the expectation of a
multiplicative functional with respect to a determinantal measure can be
expressed as a Fredholm determinant involving the correlation kernel. This
makes it possible to reformulate the claim  in terms of the correlation
operators $K_{z,z', \xi}$ and $K_{z,z'}$ (these are operators in the Hilbert
space $\ell^2(\Z')$ whose matrices are the correlation kernels of the measures
$P_{z,z',\xi}$ and $P_{z,z'}$, respectively).

The reformulation is given in Claim \ref{C4} below. Represent the Hilbert space
$H:=\ell^2(\Z')$ as the direct sum of two subspaces $H_{\pm}=\ell^2(\Z'_\pm)$
according to the splitting $\Z'=\Z'_+\sqcup\Z'_-$ (positive and negative
half--integers). Then any bounded operator in $H$ can be written as a
$2\times2$ matrix with operator entries (or ``blocks''). Let $\LL(H)$ denote
the set (actually, algebra) of bounded operators in $H$ whose two diagonal
blocks are trace class operators and two off--diagonal blocks are
Hilbert--Schmidt operators. If $\mathcal K\in\LL(H)$ then the Fredholm
determinant $\det(1+\mathcal K)$ makes sense (\cite[Appendix]{BOO}). We equip
$\LL(H)$ with the {\it combined topology\/} determined by the trace class norm
on the diagonal blocks and the Hilbert--Schmidt norm on the off--diagonal ones.

\begin{claim}\label{C4}
{\it Let $A$ stand for the operator of pointwise multiplication
by the function $x\mapsto|x|^{-1/2}$ in the space $H$. The operator
$AK_{z,z'}A$ lies in $\LL(H)$, and,  as $\xi$ goes to $1$, the operators
$AK_{z,z',\xi}A$ approach the operator $AK_{z,z'}A$ in the combined topology of
$\LL(H)$.}
\end{claim}

Note that the operator $A^2$ {\it is not\/} in $\LL(H)$, because the series
$\sum |x|^{-1}$ taken over all $x\in\Z'$ is divergent. This is the source of
difficulties. For instance, the assertion that $AK_{z,z'}A$ belongs to $\LL(H)$
is not a formal consequence of the boundedness of $K_{z,z'}$.

Claim \ref{C4} is the key technical result of the paper. The proof relies on
explicit expressions for the correlation kernels in terms of contour integrals
and requires a considerable computational work.

I do not know whether the operators $AK_{z,z',\xi}A$ approach  $AK_{z,z'}A$
simply in the trace class norm. The point is that the diagonal blocks are
Hermitian nonnegative operators while the operators themselves are not. For
nonnegative operators, one can use the fact that the convergence in the trace
class norm is equivalent to the weak convergence together with the convergence
of traces. For non--Hermitian operators, dealing with the trace class norm is
difficult, while the Hilbert--Schmidt norm turns out to be much easier to
handle. Fortunately, for the off--diagonal blocks, the convergence in the
Hilbert--Schmidt norm already suffices.

\subsection{Organization of the paper}
Section \ref{1} contains the basic notation and definitions related to the
measures under consideration and their correlation kernels. Section \ref{2}
starts with basic facts related to multiplicative functionals and their
connection to Fredholm determinants; then the proof of Claim \ref{C1} follows;
it is readily derived from the explicit expression for the first correlation
function of the measure $\unP_{z,z'}$. Section \ref{3} is devoted to the proof
of Claim \ref{C2}. In Section \ref{4}, we formulate the main result (Theorem
\ref{4.1}). Then we reduce to it to Theorem \ref{4.2} and next to Theorem
\ref{4.3}; they correspond to Claims \ref{C3} and \ref{C4}, respectively. The
main technical work is done in Sections \ref{5} and \ref{6}, where we prove
Theorem \ref{4.3} (or Claim \ref{C4}) separately for diagonal and off--diagonal
blocks.

\subsection{Acknowledgment}
I am very much indebted to Alexei Borodin for a
number of important suggestions which helped me in dealing with asymptotics of
contour integral representations. I am also grateful to Leonid Petrov and
Sergey Pirogov for valuable remarks.

\section{Z--measures and related objects}\label{1}

\subsection{Partitions and lattice point configurations}\label{1-1} A {\it
partition\/} is an infinite sequence $\la=(\la_1,\la_2,\dots)$ of nonnegative
integers $\la_i$ such that $\la_i\ge\la_{i+1}$ and only finitely many $\la_i$'s
are nonzero. We set $|\la|=\sum\la_i$. Let $\Y$ denote the set of all
partitions; it is a countable set. Following \cite{Ma}, we identify partitions
and Young diagrams.

Let $\Z'$ denote the set of all half--integers; that is, $\Z'=\Z+\frac12$. By
$\Z'_+$ and $\Z'_-$ we denote the subsets of positive and negative
half--integers, so that $\Z'$ is the disjoint union of $\Z'_+$ and $\Z'_-$.

Subsets of $\Z'$ are viewed as {\it configurations of particles\/} occupying
the nodes of the lattice $\Z'$. The unoccupied nodes are called {\it holes\/}.
Let $\Conf(\Z')$ denote the space of all particle configurations on $\Z'$. The
space $\Conf(\Z')$ can be identified with the infinite product space
$\{0,1\}^{\Z'}$ and we equip it with the product topology. In this topology,
$\Conf(\Z')$ is a totally disconnected compact space.

Recall (see Subsection \ref{0-4}) that the {\it particle/hole involution\/} on
$\Z'_-$ is the involutive map $\Conf(\Z')\to\Conf(\Z')$ keeping intact
particles and holes on $\Z'_+\subset\Z'$ and changing particles by holes and
vice versa on $\Z'_-\subset\Z'$. We denote the particle/hole involution by the
symbol ``$\inv$''. In a more formal description, ``$\inv$'' assigns to a
configuration its symmetric difference with $\Z'_-$. In particular,
$\inv(\Z'_-)=\varnothing$.

To a partition $\la\in\Y$ we assign the semi--infinite point configuration
$$
\unX(\la)=\{\la_i-i+\tfrac12\}_{i=1,2,\dots}\in\Conf(\Z').
$$
Note that among $\la_i$'s some terms may repeat while the numbers
$\la_i-i+\frac12$ are all pairwise distinct. Clearly, the correspondence
$\la\mapsto\unX(\la)$ is one--to--one. The configuration $\unX(\la)$ is
sometimes called the {\it Maya diagram\/} of $\la$, see Miwa--Jimbo--Date
\cite{MJD}.

For instance, the Maya diagram of the zero partition $\la=(0,0,\dots)$ is
$\Z'_-$. Any Maya diagram can be obtained from this one by finitely many
elementary moves consisting in shifting one particle to the neighboring
position on the right provided that it is unoccupied.

A {\it finite\/} configuration $X\subset\Z'$ is called {\it balanced\/} if
$|X\cap\Z'_+|=|X\cap\Z'_-|$. An important fact is that ``$\inv$'' establishes a
bijective correspondence between the Maya diagrams $\unX(\la)$ and the balanced
configurations. We set
$$
X(\la)=\inv(\unX(\la)), \qquad X_{\pm}(\la)=X(\la)\cap\Z'_\pm.
$$

An alternative interpretation of the balanced configuration $X(\la)$ is as
follows: $X(\la)=X_+(\la)\cup X_-(\la)$ with
$$
X_+(\la)=\{p_1<\dots<p_d\}, \quad X_-(\la)=\{-q_d<\dots<-q_1\}
$$
(recall that $|X_+(\la)|=|X_-(\la)|$), where the positive half--integers $p_i$
and $q_i$ are the {\it modified Frobenius coordinates\/} (Vershik--Kerov
\cite{VK}) of the Young diagram $\la$. They differ from the conventional
Frobenius coordinates \cite{Ma} by the additional summand $\frac12$. A slight
divergence with the conventional notation is that we arrange the coordinates in
the ascending order.

A direct explanation: $d$ is the number of diagonal boxes in $\la$ and
$$
\aligned (\la_1-1+\tfrac12,\la_2-2+\tfrac12,\dots,\la_d-d+\tfrac12)
&=(p_d,p_{d-1}\dots,p_1)\\
(\la'_1-1+\tfrac12,\la'_2-2+\tfrac12,\dots,\la'_d-d+\tfrac12)&=(q_d,q_{d-1},\dots,q_1),
\endaligned
$$
where $\la'$ is the transposed diagram.

Thus, we have defined two embeddings of the countable set $\Y$ into the space
$\Conf(\Z')$, namely, $\la\mapsto\unX(\la)$ and $\la\mapsto X(\la)$. These two
embeddings are related to each other by the particle/hole involution on
$\Z'_-$.

Note that each of the two embeddings maps $\Y$ onto a dense subset in
$\Conf(\Z')$.

\subsection{Z-measures on partitions}\label{1-2}Here we introduce a family
$\{M_{z,z',\xi}\}$ of probability measures on $\Y$, called the {\it
z--measures\/}. The subscripts $z$, $z'$, and $\xi$ are continuous parameters.
Their range is as follows: parameter $\xi$ belongs to the open unit interval
$(0,1)$,  and parameters $z$ and $z'$ should be such that $(z+k)(z'+k)>0$ for
any integer $k$. Detailed examination of this condition shows that either
$z\in\C\setminus\R$ and $z'=\bar z$ (the {\it principal series\/} of values),
or both $z$ and $z'$ are real numbers contained in an open interval $(N,N+1)$
with $N\in\Z$ (the {\it complementary series\/} of values).

We shall need the {\it generalized Pochhammer symbol\/} $(x)_\la$:
$$
(x)_\la=\prod_{i=1}^{\ell(\la)}(x-i+1)_{\la_i}\,, \qquad x\in\C, \quad
\la\in\Y,
$$
where $\ell(\la)$ is the number of nonzero coordinates $\la_i$ and
$$
(x)_k=x(x+1)\dots(x+k-1)=\frac{\Ga(x+k)}{\Ga(x)}
$$
is the conventional Pochhammer symbol. Note that
$$
(x)_\la=\prod_{(i,j)\in\la}(x+j-i),
$$
where the product is taken over the boxes $(i,j)$ of the Young diagram $\la$,
and $i$ and $j$ stand for the row and column numbers of a box.

In this notation, the weight of $\la\in\Y$ assigned by the z--measure
$M_{z,z',\xi}$ is written as
\begin{equation}\label{f1.1}
M_{z,z',\xi}(\la)=(1-\xi)^{zz'}\,\xi^{|\la|}\,(z)_\la(z')_\la\,
\left(\frac{\operatorname{dim}\la}{|\la|!}\right)^2\,,
\end{equation}
where $\operatorname{dim}\la$ is the dimension of the irreducible
representation of the symmetric group of degree $|\la|$ indexed by $\la$.

Note that $z$ and $z'$ enter the formula symmetrically, so that their
interchange does not affect the z-measure.

For the origin of formula \eqref{f1.1} and the proof that $M_{z,z',\xi}$ is
indeed a probability measure, see Borodin--Olshanski \cite{BO1}, \cite{BO2},
\cite{BO3} and references therein. Note that all the weights are strictly
positive: this follows from the conditions imposed on parameters $z$ and $z'$.
The z--measures form a deformation of the poissonized Plancherel measure  and
are a special case of Schur measures (see Okounkov \cite{Ok}).

\subsection{Limit measures}\label{1-3} Throughout the paper the parameters
$z$ and $z'$ are assumed to be fixed. If the third parameter $\xi$ approaches
0, then the z--measures $M_{z,z',\xi}$ converge to the Dirac measure at the
zero partition: this is caused by the factor $\xi^{|\la|}$.

A much more interesting picture arises as $\xi$ approaches 1. Then the factor
$(1-\xi)^{zz'}$ forces each of the weights $M_{z,z',\xi}(\la)$ to tend to 0
(note that $zz'>0$). This means that the z--measures on the discrete space $\Y$
escape to infinity. However, the situation changes when we embed $\Y$ into
$\Conf(\Z')$. Recall that we have two embeddings, one producing semi--infinite
configurations $\unX(\la)$ and the other producing finite balanced
configurations $X(\la)$. Denote by $\unP_{z,z',\xi}$ and $P_{z,z',\xi}$  the
push--forwards of the z--measure $M_{z,z',\xi}$ under these two embeddings.
Then the following result holds, see \cite{BO2}:

\begin{theorem}\label{1.1} In the space of probability measures on the compact space
$\Conf(\Z')$, there exist weak limits
$$
\unP_{z,z'}=\lim_{\xi\to1}\unP_{z,z',\xi}, \quad
P_{z,z'}=\lim_{\xi\to1}P_{z,z',\xi}.
$$
\end{theorem}

Of course, $\unP_{z,z',\xi}$ and $P_{z,z',\xi}$ are transformed to each other
under the particle/hole involution on $\Z'_-$, and the same holds for the limit
measures.

\subsection{Projection correlation kernels}\label{1-4} All the measures
appearing in Theorem \ref{1.1} are determinantal measures. Here we explain the
structure of their correlation kernels (for a detailed exposition, see
\cite{BO2}, \cite{BO3}, \cite{BO4}, and \cite{Ol2}).

The key object is a second order difference operator $D_{z,z',\xi}$ on the
lattice $\Z'$. This operator acts on a test function $f(x)$, $x\in\Z'$,
according to
$$
\aligned D_{z,z',\xi}f(x)&=\sqrt{\xi(z+x+\tfrac12)(z'+x+\tfrac12)}\,f(x+1)\\
&+\sqrt{\xi(z+x-\tfrac12)(z'+x-\tfrac12)}\,f(x-1) \\ &-[x+\xi(z+z'+x)]\,f(x).
\endaligned
$$
Since $x\pm\frac12$ is an integer for $x\in\Z'$, the expressions under the
square root are strictly positive, due to the conditions imposed on the
parameters $z$ and $z'$.

As shown in \cite{BO4}, $D_{z,z',\xi}$ determines an unbounded selfadjoint
operator in the Hilbert space $H=\ell^2(\Z')$. This operator has simple, purely
discrete spectrum filling the subset $(1-\xi)\Z'\subset\R$.

In the sequel we will freely pass from bounded operators in $H$ to their
kernels and vice versa using the natural orthonormal basis $\{e_x\}$ in $H$
indexed by points $x\in\Z'$: If $A$ is an operator in $H$ then its kernel (or
simply matrix) is defined as $A(x,y)=(Ae_y,e_x)$.

Let $\unK_{z,z',\xi}$ denote the projection in $H$ onto the positive part of
the spectrum of $D_{z,z',\xi}$, and let $\unK_{z,z',\xi}(x,y)$ denote the
corresponding kernel. (Here and below all projection operators are assumed to
be orthogonal projections.)

\begin{theorem}\label{1.2}
$\unK_{z,z',\xi}(x,y)$ is the correlation kernel of the
measure $\unP_{z,z',\xi}$.
\end{theorem}

The operator corresponding to a correlation kernel of a determinantal measure
will be called its {\it correlation operator\/}. Thus, the projection
$\unK_{z,z',\xi}$ is the correlation operator of $\unP_{z,z',\xi}$.

Let $D_{z,z'}$ denote the difference operator on $\Z'$ which is obtained by
setting $\xi=1$ in the above formula defining $D_{z,z',\xi}$. One can show that
$D_{z,z'}$ still determines a selfadjoint operator in $\ell^2(\Z')$. Its
spectrum is simple, purely continuous, filling the whole real line. Let
$\unK_{z,z'}$ denote the projection onto the positive part of the spectrum.

\begin{theorem}\label{1.3} {\rm(i)} As $\xi$ goes to $1$, the projection operators
$\unK_{z,z',\xi}$ weakly converge to a projection operator $\unK_{z,z'}$.

{\rm(ii)} $\unK_{z,z'}$ serves as the correlation operator of the limit measure
$\unP_{z,z'}$, that is, the kernel $\unK_{z,z'}(x,y)$ is the correlation kernel
of $\unP_{z,z'}$.
\end{theorem}

Note that the weak convergence of operators in $\ell^2(\Z')$ whose norms are
uniformly bounded is the same as the pointwise convergence of the corresponding
kernels. Note also that on the set of projections, the weak operator topology
coincides with the strong operator topology.

The above definition of the operators $\unK_{z,z',\xi}$ and $\unK_{z,z'}$
through the difference operators $D_{z,z',\xi}$ and $D_{z,z'}$ is nice and
useful but one often needs explicit expressions for the correlation kernels.
Various such expressions are available:

\begin{itemize}

\item Presentation in the integrable form \cite{BO1}, \cite{BO2}
$$
\frac{\mathcal A(x)\mathcal B(y)-\mathcal B(x)\mathcal A(y)}{x-y}\,.
$$

\item Series expansion (or integral representation) involving eigenfunctions of
the difference operators \cite{BO3}, \cite{BO4}.

\item Double contour integral representation \cite{BO3}, \cite{BO4}.

\end{itemize}

In Sections \ref{5} and \ref{6} we will work with contour integrals. Theorems
\ref{1.4} and \ref{1.5} below describe the integrable form for the limit kernel
$\unK_{z,z'}(x,y)$. This presentation will be used in Section \ref{2}.

\begin{theorem}\label{1.4} Assume $z\ne z'$. For $x,y\in\Z'$ and outside the
diagonal  $x=y$,
$$
\unK_{z,z'}(x,y)=\frac{\sin(\pi z)\sin(\pi z')}{\pi\sin(\pi(z-z'))}
\cdot\frac{\mathcal P(x)\mathcal Q(y)-\mathcal Q(x)\mathcal P(y)}{x-y}\,,
$$
where
\begin{equation*}
\begin{split}
\mathcal P(x) =\frac{\Ga(z+x+\tfrac12)}
{\sqrt{\Ga(z+x+\tfrac12)\Ga(z'+x+\tfrac12)}}\,, \\
\mathcal Q(x) =\frac{\Ga(z'+x+\tfrac12)}
{\sqrt{\Ga(z+x+\tfrac12)\Ga(z'+x+\tfrac12)}}
\end{split}
\end{equation*}
and $\Ga(\cdot)$ is Euler's $\Ga$--function.

On the diagonal $x=y$,
$$
\unK_{z,z'}(x,x)=\frac{\sin(\pi z)\sin(\pi z')} {\pi\sin(\pi(z-z'))}\,
(\psi(z+x+\tfrac12)-\psi(z'+x+\tfrac12)),
$$
where $\psi(x)=\Ga'(x)/\Ga(x)$ is the logarithmic derivative of the
$\Ga$--function.
\end{theorem}

See \cite{BO2} for a proof. In that paper, we called the kernel
$\unK_{z,z'}(x,y)$ the {\it Gamma kernel\/}.

In the case $z=z'$ (then necessarily $z\in\R\setminus\Z$) an explicit
expression can be obtained by taking the limit $z'\to z$ (see \cite{BO2}), and
the result is expressed through the $\psi$ function (outside the diagonal) or
its derivative $\psi'$ (on the diagonal):

\begin{theorem}\label{1.5} Assume $z=z'\in\R\setminus\Z$. For $x,y\in\Z'$ and
outside the diagonal  $x=y$,
$$
\unK_{z,z'}(x,y)=\left(\frac{\sin(\pi z)}{\pi}\right)^2\,
\frac{\psi(z+x+\tfrac12)-\psi(z+y+\tfrac12)}{x-y}\,.
$$

On the diagonal $x=y$,
$$
\unK_{z,z'}(x,x)=\left(\frac{\sin(\pi z)}{\pi}\right)^2\, \psi'(z+x+\tfrac12).
$$
\end{theorem}

Here is a simple corollary of the above formulas, which we will need later on:

\begin{corollary}\label{1.6} Let $\rho_1^{(z,z')}(x)$ denote the density function
of $P_{z,z'}$. We have
$$
\rho_1^{(z,z')}(x)\sim \frac{C(z,z')}{|x|}\,, \qquad |x|\to\infty,
$$
where
$$
C(z,z')=\begin{cases} \dfrac{\sin(\pi z)\sin(\pi z')(z-z')}
{\pi\sin(\pi(z-z'))}, & z\ne z'\\ \left(\dfrac{\sin(\pi z)}{\pi}\right)^2, &
z=z'\in\R\setminus\Z.
\end{cases}
$$
\end{corollary}

\begin{proof} Recall that $P_{z,z'}$ is related to $\unP_{z,z'}$ by the
particle/hole involution transformation on $\Z'_-$. It follows that the density
functions of the both measures coincide on $\Z'_+$. By the very definition of
determinantal measures, the density function of $\unP_{z,z'}$ is given by the
values of the correlation kernel on the diagonal $x=y$. The formulas of Theorem
\ref{1.4} and Theorem \ref{1.5} express $\unK_{z,z'}(x,x)$ through the
psi--function and its derivative. The asymptotic expansion of $\psi(y)$ as
$y\to+\infty$ is given by formula 1.18(7) in Erdelyi \cite{Er}, which implies
$$
\psi(y)=\log y-(2y)^{-1}+O(y^{-2}), \quad \psi'(y)=y^{-1}+O(y^{-2}) \qquad
(y\to+\infty).
$$
Using this we readily get
$$
\rho_1^{(z,z')}(x)=\frac{C(z,z')}{x}+O(x^{-2}), \qquad x\to+\infty,
$$

To handle the case $x\to-\infty$ one can use the relation (see \eqref{f1.2})
$$
K_{z,z'}(x,x)=1-\unK_{z,z'}(x,x), \qquad x\in\Z'_-\,,
$$
and then employ the identity (\cite[1.7.1]{Er})
$$
\psi(y+\tfrac12)-\psi(-y+\tfrac12)=\pi\tan(\pi y).
$$
A simpler way is to use the symmetry property of $P_{z,z'}$ discussed in
Subsection \ref{1-7} below. It immediately gives
$$
\rho_1^{(z,z')}(-x)=\rho_1^{(-z,-z')}(x), \qquad x\in\Z'_+\,.
$$
Since $C(-z,-z')=C(z,z')$, we get the desired formula.
\end{proof}

\begin{remark}\label{1.7}
As is seen from Theorems \ref{1.4} and \ref{1.5}, the limit kernel
$\unK_{z,z'}(x,y)$ is real-valued. The same is true for the pre--limit kernels
$\unK_{z,z',\xi}(x,y)$: this can be seen from their integrable form
presentation or from the series expansion. The fact that the kernels are
real-valued will be employed in Section \ref{6}.
\end{remark}

\subsection{$J$--Symmetric kernels and block decomposition}\label{1-5} For
technical reasons, it will be more convenient for us to deal, instead of
$\unK_{z,z',\xi}$ and $\unK_{z,z'}$, with the correlation kernels for the
measures $P_{z,z',\xi}$ and $P_{z,z'}$. The latter kernels will be denoted as
$K_{z,z',\xi}(x,y)$ and $K_{z,z'}(x,y)$, respectively. The link between two
kinds of kernels, the ``$\unK$ kernels'' and the ``$K$ kernels'', is given by
the following relation (see \cite[Appendix]{BOO} for a proof):
\begin{equation}\label{f1.2}
\epsi(x)K(x,y)\epsi(y)=\begin{cases} \unK(x,y), & x\in \Z'_+\\
\de_{xy}-\unK(x,y), & x\in\Z'_-\end{cases}\,,
\end{equation}
where
$$
\epsi(x)=\begin{cases} 1, & x\in\Z'_+\\ (-1)^{|x|-\frac12}, & x\in\Z'_-
\end{cases}.
$$

Note that the factor $\epsi(x)=\pm1$ does not affect the correlation functions
(see Subsection \ref{1-6}). This factor becomes important in the limit regime
considered in \cite{BO1} and \cite[\S8]{BO3}, but for the purpose of the
present paper, it is inessential and could be omitted; I wrote it only to keep
the notation consistent with that of the previous papers \cite{BO1},
\cite{BO2}, \cite{BO3}.

Decompose the Hilbert space $H=\ell^2(\Z')$ into the direct sum $H=H_+\oplus
H_-$, where $H_\pm=\ell^2(\Z'_\pm)$. Then every operator $A$ in $H$ can be
written in a block form,
$$
A=\bmatrix A_{++} & A_{+-}\\ A_{-+} & A_{--}
\endbmatrix\,,
$$
where $A_{++}$ acts from $H_+$ to $H_+$, $A_{+-}$ acts from $H_-$ to $H_+$,
etc.

In terms of the block form, \eqref{f1.2} can be rewritten as follows (below
$A_\epsi$ denotes the operator of multiplication by $\epsi(x)$):
\begin{align*}
K_{++}&=\unK_{++} &K_{+-}&=\unK_{+-}A_\epsi \\
K_{-+}&=-A_\epsi\unK_{-+} &K_{--}&=1-A_\epsi\unK_{--}A_\epsi.
\end{align*}
It follows that if $\unK$ is an Hermitian operator in $H$ then $K$ is also
Hermitian, but with respect to an {\it indefinite\/} inner product in $H$:
$$
[f,g]:=(Jf,g), \qquad f,g\in H, \quad J=\bmatrix 1 & 0\\0 &-1\endbmatrix.
$$
Such operators are called {\it $J$--Hermitian\/} or {\it $J$--symmetric\/}
operators. Thus, the operators $K_{z,z',\xi}$ and $K_{z,z'}$ are
$J$--symmetric.

\begin{proposition}\label{1.8} The pre--limit operators $K_{z,z',\xi}$ belong to the
trace class.
\end{proposition}

This claim is not obvious from the definition of the operators nor from the
explicit expressions for the kernels, but can be easily derived from the
results of \cite{BO1} (it is immediately seen that the ``$L$--operator''
related to $K:=K_{z,z',\xi}$ through the formula $K=L(1+L)^{-1}$ is of trace
class). The trace class property of $K_{z,z',\xi}$ is related to the fact that
the measure $P_{z,z',\xi}$ lives on finite configurations (note that the trace
of a correlation operator equals the expected total number of particles).

As for the limit measure $P_{z,z'}$, it lives on infinite configurations, and
the limit operator $K_{z,z'}$ is not of trace class.

\subsection{Gauge transformation of correlation kernels}\label{1-6}  An
arbitrary transformation of correlation kernels of the form
$$
\mathcal K(x,y)\mapsto \phi(x)\mathcal K(x,y)\phi(y)^{-1}
$$
with a nonvanishing function $\phi(x)$ does not affect the minors giving the
values of the correlation functions. We call this a {\it gauge
transformation\/}.

Thus, the correlation kernel {\it is not\/} a canonical object attached to a
determinantal measure. This circumstance must be taken into account in
attempting to solve Problem \ref{P1}.

\subsection{Symmetry}\label{1-7} Recall that by $\la\mapsto\la'$ we denote
transposition of Young diagrams. Return to formula \eqref{f1.1} for the
z--measure weights and observe that $\dim\la'=\dim\la$ and
$(z)_\la=(-1)^{|\la|}(-z)_{\la'}$. This implies the important symmetry relation
\begin{equation}\label{f1.3}
M_{z,z',\xi}(\la')=M_{-z,-z',\xi}(\la), \qquad \la\in\Y.
\end{equation}

Next, observe that under transposition $\la\mapsto\la'$, the modified Frobenius
coordinates interchange: $p_i\leftrightarrow q_i$. Together with the above
symmetry relation this implies that the transformation of the measure
$P_{z,z',\xi}$ induced by the reflection symmetry $x\mapsto-x$ of the lattice
$\Z'$ amounts to the index transformation $(z,z')\to(-z,-z')$. The same holds
for the limit measures $P_{z,z'}$.

It is worth noting that the behavior of the measures $\unP_{z,z',\xi}$ and
their limits under the reflection symmetry of $\Z'$ is more complex: besides
the change of sign of $z$ and $z'$ one has to apply the particle/hole
involution {\it on the whole lattice\/}.

The symmetry \eqref{f1.3} is reflected in the following symmetry property for
the kernels $K_{z,z',\xi}$:

\begin{proposition}\label{1.9} We have
$$
K_{z,z',\xi}(x,y)=(-1)^{\operatorname{sgn}(x)\operatorname{sgn}(y)}K_{-z,-z',\xi}(-x,-y).
$$
\end{proposition}

This follows from  \cite[Theorem 7.2]{BO3}, see the comments to this theorem.
Passing to the limit as $\xi\to1$, we get the same property for the limit
kernel $K_{z,z'}(x,y)$.

\section{Multiplicative functionals and Fredholm determinants}\label{2}

\subsection{Generalities}\label{2-1} Let $\X$ be a countable set. Below we
will return to $\X=\Z'$ but at this moment we do not need any structure on
$\X$.

As in Subsection \ref{1-1}, we mean by a {\it configuration\/} in $\X$ an
arbitrary subset $X\subseteq\X$ and we denote by $\Conf(\X)$ the space of all
configurations. Again, we equip $\Conf(\X)$ with the topology determined by the
identification $\Conf(\X)$ with the infinite product space $\{0,1\}^{\X}$; then
$\Conf(\X)$ becomes a metrizable compact topological space. We also endow
$\Conf(\X)$ with the corresponding Borel structure.

Given a configuration $X\in\Conf(\X)$, let $1_X$ denote its indicator function:
for $x\in\mathfrak X$, the value $1_X(x)$ equals 1 or 0 depending on whether
$x$ belongs or not to $X$. Viewing $1_X$ as a collection of $0$'s and $1$'s
indexed by points $x\in\X$ we just get the identification of $\Conf(\X)$ with
$\{0,1\}^{\X}$.

A function $F(X)$ on $\Conf(\X)$ is said to be a {\it cylinder function\/} if
it depends only on the intersection $X\cap Y$ with some finite subset
$Y\subset\X$. Cylinder functions are continuous and form a dense subalgebra in
$C(\Conf(\X))$, the Banach algebra of continuous functions on the compact space
$\Conf(\X)$. We will need cylinder functions in Section \ref{4}.

\subsection{Multiplicative functionals $\Phi_f$\,}\label{2-2} To a function
$f(x)$ on $\mathfrak X$, we would like to assign a {\it multiplicative
functional\/} on $\Conf(\mathfrak X)$ by means of the formula
$$
\Phi_f(X)=\prod_{x\in X}(1+f(x))=\prod_{x\in\X}(1+1_X(x)f(x)), \quad
X\in\Conf(\mathfrak X).
$$
Let us say that $\Phi_f(X)$ is {\it defined at $X$\/} if the above product is
absolutely convergent, which is equivalent to saying that the sum
\begin{equation}\label{f2.1}
\sum_{x\in X}|f(x)|=\sum_{x\in\X}1_X(x)|f(x)|
\end{equation}
is finite. Thus, the domain of definition for $\Phi_f$ is the set of all
configurations $X$ for which \eqref{f2.1} is finite.

Obviously, this set coincides with the whole space $\Conf(\X)$ if and only if
$f$ belongs to $\ell^1(\mathfrak X)$. In particular, this happens if $f$
vanishes outside a finite subset $Y\subset\X$, and then $\Phi_f$ is simply a
cylinder function. However, we will need to deal with multiplicative
functionals which are defined on a proper subset of $\Conf(\X)$ only. Observe
that for any function $f$, the domain of definition of $\Phi_f$ is a Borel
subset in $\Conf(\mathfrak X)$ (more precisely, a subset of type $F_\sigma$),
and $\Phi_f$ is a Borel function on this subset, because $\Phi_f$ is a
pointwise limit of cylinder functions.

Given a probability measure $P$ on $\Conf(\mathfrak X)$, it is important for us
to see if the domain of definition of $\Phi_f$ is of full $P$--measure. Here is
a simple sufficient condition for this, expressed in terms of the density
function $\rho_1(x)$. Recall its meaning: $\rho_1(x)$ is the probability that
the random (with respect to $P$) configuration $X$ contains $x$.

\begin{proposition}\label{2.1} Let $P$ be a probability measure on $\Conf(\mathfrak X)$,
$\rho_1(x)$ be its density function, and $f(x)$ be a function on $\mathfrak X$.
If
$$
\sum_{x\in\mathfrak X}\rho_1(x)|f(x)|<\infty
$$
then the multiplicative functional $\Phi_f(X)$ is defined $P$--almost
everywhere on\/ $\Conf(\mathfrak X)$.
\end{proposition}

\begin{proof}
Regard the quantity \eqref{f2.1} as a function in $X$, with  values in
$[0,+\infty]$. This function is almost everywhere finite if its expectation is
finite. Now, take the expectation of the right--hand side of \eqref{f2.1}.
Since the expectation of $1_X(x)$ is $\rho_1(x)$, the result is
$\sum_{x\in\mathfrak X}\rho_1(x)|f(x)|$, which is finite by the assumption.
\end{proof}

Fix a function $r(x)>0$ on $\X$. Let us say that a configuration $X$ is {\it
$r$--sparse\/} (more precisely, sparse with respect to weight $r^{-1}$) if the
series
$$
\sum_{x\in X}r^{-1}(x)=\sum_{x\in\mathfrak X}1_X(x)r^{-1}(x)
$$
converges. Let $\Conf_r(\X)$ denote the subset of all $r$--sparse
configurations. If the series $\sum_{x\in\mathfrak X}r^{-1}(x)$ converges then
obviously $\Conf_r(\X)=\Conf(\X)$; otherwise $\Conf_r(\X)$ is a proper subset
of $\Conf(\X)$. Note that it is a Borel subset (more precisely, a subset of
type $F_\sigma$).

\begin{proposition}\label{2.2}
Let $P$ be a probability measure on\/ $\Conf(\X)$ and
$\rho_1(x)$ be its density function. If
$$
\sum_{x\in\X}\rho_1(x)r^{-1}(x)<\infty
$$
then $P$ is concentrated on the Borel subset\/ $\Conf_r(\X)$.
\end{proposition}

\begin{proof}  The argument is the same as in the proof of Proposition \ref{2.1}.
Consider the function
$$
\varphi(X)=\sum_{x\in\X}1_X(x)r^{-1}(x), \qquad X\in\Conf(\X),
$$
which is allowed to take the value $+\infty$. We have to prove that
$\varphi(X)$ is finite almost surely with respect to the measure $P$. This is
obvious, because the expectation of $\varphi$ equals
$$
\sum_{x\in\X}\rho_1(x)r^{-1}(x),
$$
which is finite by the assumption.
\end{proof}

Consider the correspondence $X\mapsto m_X$ that assigns to a configuration $X$
the function $m_X(x)=1_X(x)r^{-1}(x)$ on $\X$. This correspondence determines
an embedding of the set $\Conf_r(\X)$ into the Banach space $\ell^1(\X)$. Using
this embedding we equip $\Conf_r(\X)$ with the topology induced by the norm
topology of $\ell^1(\X)$. Let us call this topology the {\it
$\ell^1$--topology\/}; of course, the definition depends on the choice of the
function $r(x)$. If the series $\sum_{x\in\mathfrak X}r^{-1}(x)$ diverges then
the $\ell^1$--topology on $\Conf_r(\X)$ is stronger than that induced by the
canonical topology of the space $\Conf(\X)$.

\begin{proposition}\label{2.3} Let $f$ be a function on $\X$ such that the function
$|f(x)|r(x)$ is bounded. Then the multiplicative functional $\Phi_f$ is well
defined on $\Conf_r(\X)$. Moreover, $\Phi_f$ is continuous in the
$\ell^1$--topology of\/ $\Conf_r(\X)$ defined above.
\end{proposition}

\begin{proof} The first claim is trivial. Indeed, set $g(x)=f(x)r(x)$. This
function is in $\ell^\infty(\X)$ while the function $m_X(x)$ is in
$\ell^1(\X)$. Since the quantity \eqref{f2.1} can be represented as the pairing
between $|g|$ and $m_X$, we conclude that \eqref{f2.1} is finite.

To prove the second claim we observe that
$$
\Phi_f(X)=\prod_{x\in\X}(1+m_X(x)g(x)).
$$
Now the claim readily follows from the fact that the pairing between
$g\in\ell^\infty(\X)$ and $m_X\in\ell^1(\X)$ is continuous in the second
argument.
\end{proof}

\begin{example}\label{2.4} Let $\X=\Z'$ and $P=P_{z,z'}$. We know from Corollary
\ref{1.6} that in this concrete case the density function $\rho_1(x)$ on $\Z'$
decays as $|x|^{-1}$ as $x\to\pm\infty$. Then Proposition \ref{2.2} says that
$P_{z,z'}$ is concentrated on $\Conf_r(\Z')$ provided that the function
$r(x)>0$ on $\Z'$ is such that the series $\sum_{x\in\Z'}r^{-1}(x)|x|^{-1}$
converges. For instance, one may take $r(x)=|x|^\de$ with any $\de>0$. (Later
on we will choose $r(x)=|x|$.)

Proposition \ref{2.3} says that a multiplicative functional $\Phi_f$ is well
defined on $\Conf_r(\Z')$ if $f(x)=O(r^{-1}(x))$ as $x\to\pm\infty$. In
particular, $\Phi_f(X)$ is well defined for $P_{z,z'}$--almost all
configurations $X$ provided that $f$ satisfies the above condition for a
positive function $r$ such that $\sum_{x\in\Z'}r^{-1}(x)|x|^{-1}<\infty$. This
condition on $f$ essentially coincides with the condition of Proposition
\ref{2.1} (only that proposition avoids the intermediation of $r$), which is
not surprising because the both propositions exploit the same idea.
\end{example}

\subsection{Condition of integrability for $\Phi_f$\,} \label{2-3} Let again
$\mathfrak X$ be a countable set and $P$ be a probability Borel measure on
$\Conf(\mathfrak X)$. Here we give a condition for $\Phi_f$ to be not only
defined $P$--almost everywhere but also to have finite expectation. The
condition involves the correlation functions of all orders. It is convenient to
combine them into a single function $\rho(X')$ defined on arbitrary finite
subsets $X'\subset\mathfrak X$: By definition, $\rho(X')$ equals the
probability of the event that the random configuration $X$ contains $X'$.

Let $X'\Subset X$ mean that $X'$ is a finite subset of $X$. By $\mathbb
E_P(\,\cdot\,)$ we denote expectation with respect to $P$.

\begin{proposition}\label{2.5}
Let $f(x)$ be a function on $\mathfrak X$ such that
$$
\sum_{X'\Subset\mathfrak X}\rho(X')\prod_{x\in X'}|f(x)|<\infty.
$$
Then the multiplicative functional $\Phi_f$ is defined almost everywhere with
respect to $P$, is absolutely integrable, and its expectation equals
\begin{equation}\label{f2.2}
\mathbb E_P(\Phi_f)=\sum_{X'\Subset\mathfrak X}\rho(X')\prod_{x\in X'}f(x).
\end{equation}
\end{proposition}

\begin{proof} The above condition on $f$ is stronger than the condition of
Proposition \ref{2.1}, so that the first claim follows from Proposition
\ref{2.1}. Checking the second and third claims uses the same argument as in
that proposition.

Observe that
$$
\prod_{x\in X}(1+|f(x)|)=\sum_{X'\Subset X}\prod_{x\in X'}|f(x)|
$$
in the sense that the both sides are simultaneously either finite or infinite,
and if they are finite then they are equal.

Denote by $\eta_{X'}(X)$ the function on $\Conf(\mathfrak X)$ equal to 1 or 0
depending on whether $X$ contains $X'$ or not. The above equality can be
rewritten as
$$
\prod_{x\in X}(1+|f(x)|)=\sum_{X'\Subset \mathfrak X}\eta_{X'}(X)\prod_{x\in
X'}|f(x)|.
$$
Take the expectation of the both sides. Since $\mathbb
E_P(\eta_{X'})=\rho(X')$, we get
$$
\mathbb E_P(\Phi_{|f|})=\sum_{X'\Subset\mathfrak X}\rho(X')\prod_{x\in
X'}|f(x)|<\infty.
$$
Thus, we have checked the second and third claims for the function $|f|$.

Since $|\Phi_f(X)|\le\Phi_{|f|}(X)$, it follows that $\Phi_f$ is absolutely
integrable. Now we can repeat the above argument with $f$ instead of $|f|$. The
above computation with $|f|$ provides the necessary justification for
manipulations with infinite sums.
\end{proof}

\subsection{Fredholm determinants} \label{2-4} Let $\mathfrak X$ and $P$ be as in
Subsection \ref{2-3}, and assume additionally that $P$ is determinantal with a
correlation kernel $K(x,y)$ corresponding to a bounded operator $K$ in the
Hilbert space $H:=\ell^2(\mathfrak X)$ (so $K$ is the correlation operator of
$P$).

For a bounded function $f(x)$ on $\mathfrak X$, we denote by $A_f$ the operator
in $H$ given by multiplication by $f$.

\begin{lemma}\label{2.6} If $f$ is finitely supported then
$$
\mathbb E_P(\Phi_f)=\det(1+A_fK).
$$
\end{lemma}

Note that the determinant is well defined because, due to the assumption on
$f$, the operator $A_fK$ has finite--dimensional range.

\begin{proof} This directly follows from \eqref{f2.2}. Indeed, according to the
definition of determinantal measures, $\rho(X')=\det\left([K(x,y)]_{x,y\in
X'}\right)$. This implies
$$
\rho(X')\prod_{x\in X'}f(x)=\det\left([f(x)K(x,y)]_{x,y\in X'}\right).
$$
Then we employ a well--known identity from linear algebra: If $B=[B(x,y)]$ is a
matrix then $\det(1+B)$ equals the sum of the principal minors of $B$. The
identity holds for matrices of finite size but we can apply it to $B=A_fK$
because the matrix $[f(x)K(x,y)]$ has only finitely many nonzero rows. This
gives us the equality
$$
\sum_{X'\Subset\mathfrak X}\det\left([f(x)K(x,y)]_{x,y\in
X'}\right)=\det(1+A_fK),
$$
which concludes the proof.
\end{proof}

The hypothesis of the lemma is, of course, too restrictive: the above argument
can be easily extended to the case when the operator $A_fK$ is of trace class.
A slightly more general fact is established below in Proposition \ref{2.7},
which is specially adapted to the application we need.

First, state a few general results from \cite[Appendix]{BOO}.

Let $H=H_+\oplus H_-$ be a $\Z_2$--graded Hilbert space. Any operator $A$ in
$H$ can be written in block form,
$$
A=\begin{bmatrix} A_{++} & A_{+-}\\ A_{-+} & A_{--}
\end{bmatrix}
$$
where $A_{++}$ acts from $H_+$ to $H_+$, $A_{+-}$ acts from $H_-$ to $H_+$,
etc. Let $\LL(H)$ be the set of bounded operators $A$ whose diagonal blocks
$A_{++}$ and $A_{--}$ are trace class operators while the off--diagonal blocks
$A_{+-}$ and $A_{-+}$ are Hilbert--Schmidt operators. The set $\LL(H)$ is an
algebra. We equip it with the corresponding combined topology: the topology of
the trace class norm $\Vert\cdot\Vert_1$ for the diagonal blocks and the
topology of the Hilbert--Schmidt norm $\Vert\cdot\Vert_2$ for the off--diagonal
blocks.

There exists a unique continuous function on $\LL(H)$,
$$
A\to\det(1+A),
$$
coinciding with the conventional determinant when $A$ is a
finite rank operator. This function can be defined as
$$
\det(1+A)=\det((1+A)e^{-A})e^{\tr(A_{++})+\tr(A_{--})},
$$
where the determinant in the right--hand side is the conventional one: the
point is that  $A\mapsto(1+A)e^{-A}-1$ is a continuous map from $\LL(H)$ to the
set of trace class operators.

If $\{E_N\}_{N=1,2,\dots}$ is an ascending chain of projection operators in $H$
strongly convergent to 1 then
\begin{equation}\label{f2.3}
\det(1+A)=\lim_{N\to\infty}(1+E_NAE_N).
\end{equation}

{}From now on we assume $\mathfrak X=\Z'$ and we set

$$
H=\ell^2(\Z'), \quad H_+=\ell^2(\Z'_+), \quad  H_-=\ell^2(\Z'_-).
$$
As before, $\{e_x\}_{x\in\Z'}$ denotes the natural basis in $H$.

\begin{proposition}\label{2.7} Let $P$ be a determinantal probability measure on
$\Conf(\Z')$,  $K(x,y)$ be its correlation kernel and $K$ denote the
corresponding correlation operator in $H=\ell^2(\Z')$.

Further, assume that $f(x)$ is a function on $\Z'$ which can be written in the
form $f(x)=g(x)h^2(x)$, where $g(x)$ is bounded and $h$ is nonnegative and such
that $A_h K A_h\in\LL(H)$.

Then the functional $\Phi_f$ is defined almost everywhere with respect to $P$,
is absolutely integrable with respect to $P$, and
$$
\mathbb E_P(\Phi_f)=\det(1+A_gA_h K A_h).
$$
\end{proposition}

\begin{proof} For $N=1,2,\dots$, let $E_N$ be the projection in $H$ onto the
finite--dimensional subspace of functions concentrated on $[-N,N]\cap\Z'$. As
$N\to\infty$, the projections $E_N$ converge to 1.

Assume first that $g\equiv1$. Then $f(x)$ is nonnegative and the same argument
as in Lemma \ref{2.6} shows that
\begin{gather*}
\det(1+E_NA_h K A_hE_N)=\det(1+A_fE_NKE_N)\\
=\sum_{X'\subset[-N,N]\cap\Z'}\det\left([K(x,y)]_{x,y\in X'}\right)\prod_{x\in X'}f(x)\\
= \sum_{X'\subset[-N,N]\cap\Z'}\rho(X')\prod_{x\in X'}f(x),
\end{gather*}
where we used the fact that $E_N$ and $A_h$ commute. As $N\to\infty$, the
resulting quantity converges to the infinite sum
$$
\sum_{X'\Subset\Z'}\rho(X')\prod_{x\in X'}f(x).
$$
On the other hand, by virtue of \eqref{f2.3},
$$
\det(1+E_NA_gA_h K A_hE_N)\to\det(1+A_h K A_h).
$$
Consequently,
$$
\det(1+A_hKA_h)=\sum_{X'\Subset\Z'}\rho(X')\prod_{x\in X'}f(x)=\mathbb
E_P(\Phi_f),
$$
where the last equality follows from Proposition \ref{2.5}.

For an arbitrary bounded $g$ the computation is the same, and the above
argument with $f\ge0$ is used to check the absolute convergence of the arising
infinite sum and to guarantee applicability of Proposition \ref{2.5}.
\end{proof}

\section{Radon--Nikod\'ym derivatives}\label{3}

Denote by $\mathfrak S$ the group of finitary permutations of the set $\Z'$.
This is a countable group generated by the elementary transpositions
$$
\dots,\si_{-1},\si_0,\si_1,\dots,
$$
where $\si_n$ transposes $n-\frac12$ and
$n+\frac12$. The action of the group $\mathfrak S$ on $\Z'$ induces its action
on $\Conf(\Z')$. For $\si\in\mathfrak S$, we denote the corresponding
transformation of $\Conf(\Z')$ by the same symbol $\si$.

Let us represent configurations $X\in\Conf(\Z')$ as two--sided infinite
sequences of black and white circles separated by vertical bars, like this:
$$
\cdots\bullet\mathop{|}_{n-2}\bullet
\mathop{|}_{n-1}\circ\mathop{|}_{n}\bullet\mathop{|}_{n+1}\circ
\mathop{|}_{n+2}\circ\cdots
$$

Here black and white circles represent particles and holes, respectively, and
the subscripts under the bars are used to mark the positions of integers
interlacing with half--integers. In this picture, the action of $\si_n$ affects
only the two--circle fragment around the $n$th vertical bar and amounts to
replacing  ``$\circ\mathop{|}\limits_{n}\bullet$'' by
``$\bullet\mathop{|}\limits_{n}\circ$'' and vice versa (the fragments
``$\circ\mathop{|}\limits_{n}\circ$'' and
``$\bullet\mathop{|}\limits_{n}\bullet$'' remain intact).

This action preserves the set of Maya diagrams $\unX(\la)$, so that we get an
action of $\mathfrak S$ on $\Y$, which can be directly described as follows:
application of the elementary transposition $\si_n$ to a Young diagram $\la$
amounts to adding or removing a box $(i,j)$ with $j-i=n$, if this operation is
possible. The transformation ``$\bullet\mathop{|}\limits_{n}\circ\to
\circ\mathop{|}\limits_{n}\bullet$'' corresponds to adding a box, and the
inverse transformation corresponds to removing a box. In particular, $\si_0$
adds/removes boxes on the main diagonal of $\la$.

Our aim is to study the transformation of the measures $\unP_{z,z',\xi}$ and
their limits $\unP_{z,z'}$ under the action of $\mathfrak S$. However, for
technical reasons, it is more convenient to deal with the measures
$P_{z,z',\xi}$ and $P_{z,z'}$ which are related to the former measures by the
particle/hole involution on $\Z'_-$. To this end we introduce the {\it modified
action\/} of $\mathfrak S$ on $\Conf(\Z')$: it differs from the natural one by
conjugation with the particle/hole involution. Given $\si\in\mathfrak S$, we
denote the modified action of $\si$ on $\Conf(\Z')$ by the symbol $\tsi$. The
relation between $\si$ and $\tsi$ is
$$
\tsi(X)=\inv(\si(\inv(X))), \qquad \si\in\mathfrak S, \quad X\in\Conf(\Z').
$$
The modified transformations $\tsi:\Conf(\Z')\to\Conf(\Z')$ induce
transformations of measures denoted as $P\mapsto \tsi(P)$.

Let us emphasize that the modified action is defined only on $\Conf(\Z')$, not
on $\Z'$ itself.

In the case of elementary transpositions $\si=\si_n$, the modified action
differs from the natural one for $n=0$ only. Namely, the modified action of
$\si_0$ amounts to switching ``$\circ\mathop{|}\limits_{0}\circ\leftrightarrow
\bullet\mathop{|}\limits_{0}\bullet$'', while the fragments
``$\bullet\mathop{|}\limits_{0}\circ$'' and
``$\circ\mathop{|}\limits_{0}\bullet$'' remain intact.

Recall that a finite configuration $X\Subset\Z'$ has the form $\inv(\unX(\la))$
with $\la\in\Y$ if and only if $X$ is {\it balanced\/} in the sense that
$|X\cap\Z'_+|=|X\cap\Z'_-|$. Since the initial action of $\mathfrak S$
preserves the set of the semi--infinite configurations of the form $\unX(\la)$,
the modified action preserves the set of the finite balanced configurations.
Obviously, if $\si\in\mathfrak S$,  $\unX=\unX(\la)$, and
$X=X(\la)=\inv(\unX)$, then we have
$$
\frac{\si(\unP_{z,z',\xi})(\unX)}{\unP_{z,z',\xi}(\unX)}
=\frac{\tsi(P_{z,z',\xi})(X)}{P_{z,z',\xi}(X)}\,.
$$
We introduce a special notation for this Radon--Nikod\'ym derivative:
\begin{equation}\label{f3.1}
\mu_{z,z',\xi}(\si,X):=\frac{\tsi(P_{z,z',\xi})(X)}{P_{z,z',\xi}(X)}
=\frac{P_{z,z',\xi}(\tsi^{-1}(X))}{P_{z,z',\xi}(X)}, \qquad \si\in\mathfrak S,
\end{equation}
where $X$ is a finite balanced configuration.

In the remaining part of the section we prove the following result.

\begin{proposition}\label{3.1} Fix an arbitrary couple $(z,z')$ of parameters
belonging to the principal or complementary series. Let $\xi$ range over
$(0,1)$ and $X$ range over the set of finite balanced configurations on $\Z'$.

For any fixed $\si\in\mathfrak S$, the Radon--Nikod\'ym derivative \eqref{f3.1}
can be written as a finite linear combination of multiplicative functionals of
the form $\Phi_f$ multiplied by factors $\xi^k$ with $k\in\Z$, where each
function $f(x)$ decays at infinity at least as $|x|^{-1}${\rm:}
\begin{equation}\label{f3.2}
\begin{split}
\mu_{z,z',\xi}(\si,X)=\sum_{i=1}^m a_i\xi^{k_i}\Phi_{f_i}(X), \\
a_i\in\R, \quad k_i\in\Z, \quad f_i(x)=O(|x|^{-1}).
\end{split}
\end{equation}
\end{proposition}

Let us emphasize that the right--hand side depends on $\xi$ through the factors
$\xi^k$ only. Proposition \ref{3.1} provides a refinement of Claim \ref{C2} of
the Introduction, as explained after the end of the proof of the proposition.

\begin{proof}
{\it Step\/} 1. Given a subset $X'\subseteq\Z'\cap[-N,N]$, where
$N\in\{1,2,\dots\}$, and a function $f^\out$ on $\Z'\setminus[-N,N]$, we set
\begin{equation*}
\begin{split}
\eta_{N,X'}(X)=\begin{cases} 1, & X\cap[-N,N]=X'\\ 0, &\text{\rm
otherwise}\end{cases}, \\
\Phi_{N,f^\out}(X)=\prod_{x\in
X\setminus[-N,N]}(1+f^\out(x)).
\end{split}
\end{equation*}

We will prove that for any $\si\in\mathfrak S$ and all $N$ large enough there
exists a representation of the form
\begin{equation}\label{f3.3}
\begin{gathered}
\mu_{z,z',\xi}(\si,X)=\sum_{i=1}^m a_i\xi^{k_i}\eta_{N,X'_i}(X)\Phi_{N,f^\out_i}(X), \\
a_i\in\R, \quad k_i\in\Z, \quad X'_i\subseteq\Z'\cap[-N,N], \quad
f^\out_i(x)=O(|x|^{-1}).
\end{gathered}
\end{equation}
Observe that \eqref{f3.3}  is  equivalent to \eqref{f3.2}, because each
function of the form $\eta_{N,X'}(X)$, being a cylinder functional depending on
$X\cap[-N,N]$ only, can be written as a linear combination of functionals of
the form
$$
X\mapsto\prod_{x\in X\cap[-N,N]}(1+f^\inn(x))
$$
with appropriate functions $f^\inn$ on $\Z'\cap[-N,N]$.

{\it Step\/} 2. Next, we want to reduce the problem to the particular case when
$\si$ is an elementary transposition. Since the elementary transpositions
generate the whole group $\mathfrak S$, to perform the desired reduction, it
suffices to prove that if the presentation \eqref{f3.3} exists for two elements
$\si,\tau\in\mathfrak S$ then it also exists for the product $\si\tau$.

It follows from the definition \eqref{f3.1} that
$$
\mu_{z,z',\xi}(\si\tau,X) =\mu_{z,z',\xi}(\si,X)\cdot
\mu_{z,z',\tau}(\tau,\wt\si^{-1}(X)).
$$
We may assume that $N$ is so large that the permutation $\si:\Z'\to\Z'$ does
not move points outside $[-N,N]$. Then it is clear that if the function
$X\mapsto \mu_{z,z',\xi}(\tau,X)$ admits a presentation of the form
\eqref{f3.3} then the same holds for the function $X\mapsto
\mu_{z,z',\xi}(\tau,\wt\si^{-1}(X))$ as well. Thus, it remains to check that
the set of functions admitting a representation of the form \eqref{f3.3} (with
$N$ fixed) is closed under multiplication. This is obvious, because the product
of two functionals of the form $\Phi_{N,f^\out}$ is a functional of the same
kind:
$$
\Phi_{N,f^\out}\Phi_{N,g^\out}=\Phi_{N,h^\out}
$$
with
$$
h^\out(x):=f^\out(x)g^\out(x)+f^\out(x)+g^\out(x),
$$
and, moreover, if $f^\out(x)=O(|x|^{-1})$ and $g^\out(x)=O(|x|^{-1})$ then
$h^\out(x)=O(|x|^{-1})$.

{\it Step\/} 3. Thus, we have to analyze the ratio
\begin{equation}\label{f3.4}
\mu_{z,z',\xi}(\si_n,X) =\frac{P_{z,z',\xi}(\tsi_n^{-1}(X))}{P_{z,z',\xi}(X)}
=\frac{P_{z,z',\xi}(\tsi_n(X))}{P_{z,z',\xi}(X)}\,,
\end{equation}
where the second equality holds because $\tsi_n^{-1}=\tsi_n$.

We aim to prove that for any $N>|n|$, there exists a single term representation
\begin{equation}\label{f3.5}
\begin{split}
\frac{P_{z,z',\xi}(\tsi_n(X))}{P_{z,z',\xi}(X)}=a\xi^k\Phi_{N,f^\out}(X),\\
k=0,\pm1, \quad f^\out(x)=O(|x|^{-1}),
\end{split}
\end{equation}
where, in contrast to \eqref{f3.3}, $a$, $k$, and $f^\out$ may depend on the
intersection $X':=X\cap[-N,N]$. Observe that once \eqref{f3.5} is established,
we can combine various variants of \eqref{f3.5} (which depend on $X'$) into a
single representation \eqref{f3.3} by making use of the factors
$\eta_{N,X'}(X)$. Thus, it suffices to prove \eqref{f3.5}.

{\it Step\/} 4. Here we exhibit a convenient explicit expression for
$P_{z,z',\xi}(X)$, where $X$ is an arbitrary balanced configuration. Employing
the notation introduced in Subsection \ref{1-1} we write
$$
X=\{-q_d,\dots,-q_1,p_1,\dots,p_d\}.
$$
By definition, $P_{z,z',\xi}(X)=M_{z,z',\xi}(\la)$, where $\la$ is such that
$X=X(\la)$. We have to rewrite the expression \eqref{f1.1} for
$M_{z,z',\xi}(\la)$ given in Subsection \ref{1-2} in terms of the $p_i$'s and
$q_i$'s. For the terms $(z)_\la$ and $(z')_\la$ this is easy, and for
$\dim\la/|\la|!$ we employ the formula
$$
\frac{\dim\la}{|\la|!}=\frac{\prod\limits_{1\le i<j\le
d}(p_j-p_i)(q_j-q_i)}{\prod\limits_{i=1}^d(p_i-\frac12)!(q_i-\frac12)!\cdot
\prod\limits_{i,j=1,\dots,d}(p_i+q_j)}
$$
given, e.g., in \cite[(2.7)]{Ol1}. The result is
\begin{equation}\label{f3.6}
\begin{split}
&P_{z,z',\xi}(X)=(1-\xi)^{zz'}\xi^{\sum_{i=1}^d(p_i+q_i)}(zz')^d \\
&\quad\times\prod_{i=1}^d \frac{(z+1)_{p_i-\frac12}(z'+1)_{p_i-\frac12}
(-z+1)_{q_i-\frac12}(-z'+1)_{q_i-\frac12}}{((p_i-\frac12)!)^2((q_i-\frac12)!)^2}\\
&\quad\times\frac{\prod\limits_{1\le i<j\le
d}(p_j-p_i)^2(q_j-q_i)^2}{\prod\limits_{i,j=1,\dots,d}(p_i+q_j)^2}.
\end{split}
\end{equation}
At first glance formula \eqref{f3.6} might appear cumbersome but actually it is
well suited for our purpose, because it already has multiplicative form and
after substitution into \eqref{f3.4} many factors are cancelled out.

Below we examine separately the three cases: $n=1,2,\dots$, $n=-1,-2,\dots$,
and $n=0$.

{\it Step\/} 5. Consider the case $n=1,2,\dots$. Then the transformed
configuration $\tsi_n(X)$ is the same as the configuration $\si_n(X)$, which in
turn either coincides with the initial configuration $X$ or differs from it by
shifting a single coordinate $p_i$ by $\pm1$. The shift $p_i\to p_i+1$ arises
if there exists $i$ such that $p_i=n-\frac12$ and either $i=d$ or $p_i+1\ne
p_{i+1}$, which means that $X$ contains the fragment
$\bullet\underset{n}{|}\circ$. The shift $p_i\to p_i-1$ arises if there exists
$i$ such that $p_i=n+\frac12$ and either $i=1$ or $p_i-1\ne p_{i-1}$, which
means that $X$ contains the fragment $\circ\underset{n}{|}\bullet$. In all
other cases $\si_n(X)=X$. Clearly, what of these possible variants takes place
is uniquely determined by the intersection $X':=X\cap[-N,N]$ (recall that, by
assumption, $N>|n|$).

If $\si_n(X)=X$ then the ratio \eqref{f3.4} simply equals 1.

If $\si_n$ transforms $p_i$ to $p_i\pm1$ then, as directly follows from
\eqref{f3.6}, the ratio \eqref{f3.4} equals
$$
\xi^{\pm1}\,
\left(\dfrac{(z+p_i\pm\tfrac12)(z'+p_i\pm\tfrac12)}{(p_i\pm\tfrac12)^2}\right)^{\pm1}
\dfrac{\prod\limits_{j:\, j\ne i}\left(\dfrac{p_j-p_i\mp1}{p_j-p_i}\right)^2}
{\prod\limits_j\left(\dfrac{q_j+p_i\pm1}{q_j+p_i}\right)^2}\,.
$$
This has the desired form \eqref{f3.5} with $k=\pm1$ and
$$
1+f^\out(x)=\begin{cases} \left(1\mp\dfrac1{x-p_i}\right)^2, & x>N\\
\left(1\pm\dfrac1{|x|+p_i}\right)^{-2}, & x<-N. \end{cases}
$$

{\it Step\/} 6. In the case $n=-1,-2,\dots$ one can repeat the argument of step
5. Alternatively, one can use the symmetry $p_i\leftrightarrow q_i$ (Subsection
\ref{1-7}).

{\it Step\/} 7. Finally, consider the case $n=0$. Then either $\tsi_0(X)=X$ or
$\tsi_0(X)$ differs from $X$ by adding or removing the couple of coordinates
$p_1=\frac12$, $q_1=\frac12$. Therefore,  ratio \eqref{f3.4} either equals 1 or
has the form
$$
(zz')^{\pm1}\xi^{\pm1}\left(\prod_j\dfrac{(p_j-\tfrac12)(q_j-\tfrac12)}
{(p_j+\tfrac12)(q_j+\tfrac12)}\right)^{\pm2}.
$$
This has the desired form \eqref{f3.5} with $k$ equal to $0$ or $\pm1$ and
$f^\out$ equal to 0 or
$$
1+f^\out(x)=\left(1-\frac1{2|x|}\right)^{\pm2}\left(1+\frac1{2|x|}\right)^{\mp2}.
$$
\end{proof}

In the discussion below we use the notions introduced in Subsection \ref{2-2}.

\begin{definition}\label{3.2}
Take the function $r(x)=|x|$ on $\Z'$. The corresponding
subset $\Conf_r(\Z')\subset\Conf(\Z')$ of $r$--sparse configurations will be
denoted as  $\Conf_\sparse(\Z')$. We equip $\Conf_\sparse(\Z')$ with the
$\ell^1$--topology.
\end{definition}

By virtue of Proposition \ref{2.3}, any function of the form \eqref{f3.2} is
well defined and continuous on $\Conf_\sparse(\Z')$. Thus, for any
$\si\in\mathfrak S$, the function $\mu_{z,z',\xi}(\si,X)$, initially defined on
finite balanced configurations $X$, admits a continuous extension to the larger
set $\Conf_\sparse(\Z')$. Although the presentation \eqref{f3.2} is not unique,
the result of the continuous extension provided by formula \eqref{f3.2} does
not depend on the specific presentation. Indeed, this follows from the fact
that the finite balanced configurations form a dense subset in
$\Conf_\sparse(\Z')$ (which is readily checked).

\begin{definition}\label{3.3}
For any $\si\in\mathfrak S$, let $\mu_{z,z'}(\si,X)$
stand for the function on $\Conf_\sparse(\Z')$ obtained by specializing $\xi=1$
in the right--hand side of formula \eqref{f3.2}.
\end{definition}

Obviously, $\mu_{z,z'}(\si,X)$ coincides with the pointwise limit, as
$\xi\to1$, of the continuous extensions of the functions
$\mu_{z,z',\xi}(\si,X)$. This shows that $\mu_{z,z'}(\si,X)$ does not depend on
a specific presentation \eqref{f3.2}. Moreover, Proposition \ref{2.3} ensures
that the limit function is continuous in the $\ell^1$--topology.

Thus, we have shown that Proposition \ref{3.1} implies Claim \ref{C2}
(Subsection \ref{0-4}) in a refined form.

\section{Main result: Formulation and beginning of proof}\label{4}

The following theorem is the main result of the paper.

\begin{theorem}\label{4.1} {\rm(i)} The measures $P_{z,z'}$ are quasiinvariant with
respect to the modified action of the group $\mathfrak S$ on probability
measures on the space $\Conf(\Z')$, as defined in Section \ref{3}.

{\rm(ii)} For any permutation $\si\in\mathfrak S$, the Radon--Nikod\'ym
derivative $\tsi(P_{z,z'})/P_{z,z'}$ coincides with the limit expression
$\mu_{z,z'}(\si,X)$ introduced in Definition \ref{3.3}, within a
$P_{z,z'}$--null set.
\end{theorem}

Recall that each of the measures $P_{z,z'}$ is concentrated on the Borel subset
$\Conf_\sparse(\Z')$ (see Example \ref{2.4}) and each of the functions
$\mu_{z,z'}(\si,X)$ is well defined on the same subset and is a Borel function.

In this section, we will reduce Theorem \ref{4.1} to Theorem \ref{4.3} through
an intermediate claim, Theorem \ref{4.2}, which is of independent interest. The
proof of Theorem \ref{4.3} occupies Sections \ref{5} and \ref{6}.

As before, we use the angular brackets to denote the pairing between functions
and measures.

\begin{theorem}\label{4.2} Let $f(x)$ be an arbitrary function on $\Z'$ such that
$f(x)=O(|x|^{-1})$ as $x\to\pm\infty$. Then the multiplicative functional
$\Phi_f$ is absolutely integrable with respect to the measures $P_{z,z',\xi}$
and $P_{z,z'}$ and
$$
\lim_{\xi\to1}\langle\Phi_f,\, P_{z,z',\xi}\rangle=\langle\Phi_f,\,
P_{z,z'}\rangle.
$$
\end{theorem}

\begin{proof}[Derivation of Theorem \ref{4.1} from Theorem \ref{4.2}]
Recall that the function $\mu_{z,z'}(\si,X)$ is a finite linear combination of
the multiplicative functionals $\Phi_f$ with $f(x)=O(|x|^{-1})$, see Definition
\ref{3.3}. Theorem \ref{4.2} says that such functionals are absolutely
integrable with respect to $P_{z,z'}$, which implies that so is
$\mu_{z,z'}(\si,X)$. Thus, $\mu_{z,z'}(\si,\,\cdot\,)P_{z,z'}$ is a finite
Borel measure.

The claim of Theorem \ref{4.1} is equivalent to the following one: For any
$\si\in\mathfrak S$,
$$
\tsi(P_{z,z'})=\mu_{z,z'}(\si,\,\cdot\,)P_{z,z'}.
$$
Recall that the cylinder functions on $\Conf(\Z')$ are dense in
$C(\Conf(\Z'))$, see Subsection \ref{2-1}. Therefore, it suffices to prove that
for any cylinder function $F$
$$
\langle F, \,\tsi(P_{z,z'})\rangle=\langle \mu_{z,z'}(\si,\,\cdot\,)F,\,
P_{z,z'}\rangle.
$$

Set
$$
F^{\si}(X)=F(\tsi(X))
$$
and observe that $F^\si$ is a cylinder function, too. We may rewrite the
desired equality in the form
\begin{equation}\label{f4.1}
\langle F^\si,\, P_{z,z'}\rangle=\langle \mu_{z,z'}(\si,\,\cdot\,)F,\,
P_{z,z'}\rangle.
\end{equation}

By the very definition of $\mu_{z,z',\xi}$, we have
$$
\tsi(P_{z,z',\xi})=\mu_{z,z',\xi}(\si,\,\cdot\,)P_{z,z',\xi},
$$
so that
\begin{equation}\label{f4.2}
\langle F^\si,\, P_{z,z',\xi}\rangle=\langle \mu_{z,z',\xi}(\si,\,\cdot\,)F,\,
P_{z,z',\xi}\rangle.
\end{equation}
A natural idea is to derive \eqref{f4.1} from \eqref{f4.2} by passing to the
limit as $\xi$ goes to 1.

We know that the measures $P_{z,z',\xi}$ weakly converge to the measure
$P_{z,z'}$ (Theorem \ref{1.1} above). Therefore, the left--hand side of
\eqref{f4.2} converges to the left--hand side of \eqref{f4.1}.

Consequently, to establish \eqref{f4.1} it remains to prove that the similar
limit relation holds for the right--hand sides, namely
\begin{equation}\label{f4.3}
\lim_{\xi\to1}\langle \mu_{z,z',\xi}(\si,\,\cdot\,)F,\, P_{z,z',\xi}\rangle =
\langle \mu_{z,z'}(\si,\,\cdot\,)F,\, P_{z,z'}\rangle.
\end{equation}

According to the definition of the function $\mu_{z,z'}(\si,X)$ (see Theorem
\ref{3.1} and Definition \ref{3.3}), the limit relation \eqref{f4.3} can be
reduced to the following one:
\begin{equation}\label{f4.4}
\lim_{\xi\to1}\langle \xi^k\Phi_f F,\, P_{z,z',\xi}\rangle =\langle \Phi_f F,\,
P_{z,z'}\rangle,
\end{equation}
where $k\in\Z$, $F$ is a cylinder function, and $f(x)=O(|x|^{-1})$.

Obviously, the factor $\xi^k$, which tends to $1$, is inessential and can be
neglected, so that \eqref{f4.4} can be simplified:
\begin{equation}\label{f4.5}
\lim_{\xi\to1}\langle \Phi_f F,\, P_{z,z',\xi}\rangle =\langle \Phi_f F,\,
P_{z,z'}\rangle.
\end{equation}

Next, fix a finite subset $Y\subset\X$, so large that $F(X)$ depends on the
intersection $X\cap Y$ only. It is readily verified that $F$ can be written as
a finite linear combination of multiplicative functionals $\Phi_{g_i}$, where
each $g_i$ vanishes outside $Y$ (this claim actually concerns functions on the
finite set $\{0,1\}^Y$). Observe that
$$
\Phi_f\Phi_{g_i}=\Phi_{f_i}\,, \qquad f_i:=f+g_i+fg_i
$$
(we have already used such an equality in the proof of Proposition \ref{3.1},
step 2). It follows that the product $\Phi_f F$ can be written as a finite
linear combination of multiplicative functionals $\Phi_{f_i}$, where each $f_i$
coincides with $f$ outside $Y$ and hence obeys the same decay condition,
$f_i(x)=O(|x|^{-1})$. Thus, we have reduced \eqref{f4.5} to the claim of
Theorem \ref{4.2}.
\end{proof}

The essence of difficulty in proving Theorem \ref{4.2} is that, for generic $f$
decaying as $|x|^{-1}$, the multiplicative functional $\Phi_f$  is unbounded
and so cannot be extended to a continuous function on the whole space
$\Conf(\Z')$. Thus, for our purpose, the fact of the weak convergence
$P_{z,z',\xi}\to P_{z,z'}$, that is, convergence on continuous test functions,
is insufficient: we have to enlarge the set of admissible test functions to
include the functions like $\Phi_f$. \footnote{The situation is formally
similar to that of weak convergence and moment convergence of probability
measures on $\R$: In general, the former does not imply the latter.}

The idea is to relate the required stronger convergence of the measures to an
appropriate convergence of their correlation operators.

Set $h(x)=|x|^{-1/2}$, where $x\in\Z'$, and recall that $A_h$ denotes the
operator of multiplication by $h$ in the Hilbert space $H=\ell^2(\Z')$. Below
we use the notions introduced in Subsection \ref{2-4}.

\begin{theorem}\label{4.3} {\rm(i)} The operator $A_hK_{z,z'}A_h$ lies in
$\LL(H)$.

{\rm(ii)} As $\xi$ goes to $1$, the operators $A_hK_{z,z',\xi}A_h$ converge to
the operator $A_hK_{z,z'}A_h$ in the topology of the space $\LL(H)$.
\end{theorem}

\medskip
\noindent{\it Comments\/.} Note that the pre--limit operators
$A_hK_{z,z',\xi}A_h$ also lie in $\LL(H)$, because the operators $K_{z,z',\xi}$
are of trace class, see Proposition \ref{1.8}. According to the definition of
$\LL(H)$ (see Subsection \ref{2-4}), the claim of the theorem means that the
diagonal blocks converge in the trace class norm while the off--diagonal blocks
converge in the Hilbert--Schmidt norm. I do not know whether, in the case of
the off--diagonal blocks, the Hilbert--Schmidt norm can be replaced by the
trace class norm.
\medskip

\begin{proof}[Derivation of Theorem \ref{4.2} from Theorem \ref{4.3}]
The assumption on the function $f$ allows one to write it as $f=gh^2$, where
$|g(x)|$ is bounded (recall that $h(x)=|x|^{-1/2}$). By virtue of Proposition
\ref{2.7} and claim (i) of Theorem \ref{4.3}, $\Phi_f$ is absolutely integrable
with respect to the limit measure $P_{z,z'}$, and the same also holds for the
pre--limit measures $P_{z,z',\xi}$ (see the comments above). Moreover,
$$
\langle \Phi_f,\, P_{z,z',\xi}\rangle=\det(1+A_gA_hK_{z,z',\xi}A_h), \qquad
\langle \Phi_f,\, P_{z,z'}\rangle=\det(1+A_gA_hK_{z,z'}A_h).
$$

Finally, claim (ii) of Theorem \ref{4.2} implies that the operators
$A_gA_hK_{z,z',\xi}A_h$ converge to the operator $A_gA_hK_{z,z'}A_h$ in the
topology of $\LL(H)$. Therefore, the corresponding determinants converge, too.
Here we use the fact that the function $A\mapsto\det(1+A)$ is continuous on
$\LL(H)$, see Subsection \ref{2-4}.
\end{proof}

Thus, we have reduced Theorem \ref{4.2}, and hence Theorem \ref{4.1}, to
Theorem \ref{4.3}. The latter theorem is proved separately for the diagonal and
off--diagonal blocks in Sections \ref{5} and \ref{6}, respectively.

\section{Convergence of diagonal blocks in the topology of the trace class
norm}
\label{5}

Here we prove the claims of Theorem \ref{4.3} for the diagonal blocks
$(\cdot)_{++}$ and $(\cdot)_{--}$ of the operators in question. Due to the
symmetry relation of Proposition \ref{1.9}, the latter block is obtained from
the former one by a simple change of the basic parameters, $(z,z')\to(-z,-z')$.
Thus, it suffices to focus on the limit behavior of the block $(\cdot)_{++}$.
That is, we have to prove that the operator $(A_hK_{z,z'}A_h)_{++}$ in the
Hilbert space $\ell^2(\Z'_+)$ is of trace class and
$$
\lim_{\xi\to1}\Vert (A_hK_{z,z',\xi}A_h)_{++}-(A_hK_{z,z'}A_h)_{++}\Vert_1=0,
$$
where $\Vert\cdot\Vert_1$ is the trace class norm.

\medskip
Since the proof is long, let us describe its scheme. First of all, observe
that, as long as we are dealing with the $++$ block, there is no difference
between $K_{z,z',\xi}$ and $\unK_{z,z',\xi}$ (and the same for the limit
operators). Indeed, this follows from the relation between the both kind of
operators, see Subsection \ref{1-5}.

In Proposition \ref{5.1} we rederive the weak convergence $\unK_{z,z',\xi}\to
\unK_{z,z'}$, which implies the weak convergence
$(A_hK_{z,z',\xi}A_h)_{++}\to(A_hK_{z,z'}A_h)_{++}$.

Recall that $\unK_{z,z',\xi}$ is a projection operator (see Subsection
\ref{1-4}). This implies that its $++$ block is a nonnegative operator, so that
$(A_hK_{z,z',\xi}A_h)_{++}$ is also a nonnegative operator. Therefore, the
limit operator $(A_hK_{z,z'}A_h)_{++}$ is nonnegative, too.

Consequently, to prove that the operator $(A_hK_{z,z'}A_h)_{++}$ is of trace
class it suffices to prove that its trace is finite. This is done in
Proposition \ref{5.2}.

The similar fact for the pre--limit operator $(A_hK_{z,z',\xi}A_h)_{++}$ is a
trivial consequence of Proposition \ref{1.8}.

In Proposition \ref{5.4} we establish the convergence of traces,
$$
\lim_{\xi\to1}\tr\left((A_hK_{z,z',\xi}A_h)_{++}\right)
=\tr\left((A_hK_{z,z'}A_h)_{++}\right).
$$
This concludes the proof, because, for nonnegative operators, weak convergence
together with convergence of traces is equivalent to convergence in the trace
class norm (see, e.g., \cite[Proposition A.9]{BOO}).
\medskip

Let us proceed to the detailed proof.

Our starting point is the double contour integral representation \eqref{f5.1}
(see below) for the kernel $\unK_{z,z',\xi}(x,y)$ on the lattice $\Z'$. Formula
\eqref{f5.1} is a particular case of a more general formula obtained in
\cite[\S9, p. 148]{BO3}. \footnote{Another integral representation, given in
\cite[Thm.3.3]{BO4} and \cite[Thm. 6.3]{BO3}, seems to be less suitable for our
purpose.}
\begin{multline}\label{f5.1}
\unK_{z,z',\xi}(x;y)\\
=\frac{\Gamma(-z'-x+\frac12)\Gamma(-z-y+\frac12)}
{\bigl(\Gamma(-z-x+\frac12)\Gamma(-z'-x+\frac12)
\Gamma(-z-y+\frac12)\Gamma(-z'-y+\frac12)\bigr)^{\frac12}} \\
\times\frac{1-\xi}{(2\pi i)^2}\oint\limits_{\{\om_1\}}\oint\limits_{\{\om_2\}}
\left(1-\sqrt{\xi}\om_1\right)^{z'+x-\tfrac12}
\left(1-\frac{\sqrt{\xi}}{\om_1}\right)^{-z-x-\tfrac12}\\
\times\left(1-\sqrt{\xi}\om_2\right)^{z+y-\tfrac12}
\left(1-\frac{\sqrt{\xi}}{\om_2}\right)^{-z'-y-\tfrac12} \frac{\om_1^{-x-\frac
12}\om_2^{-y-\frac 12}}{\omega_1\omega_2-1} \,{d\om_1}{d\om_2}.
\end{multline}

Let us explain the notation. Here $\{\om_1\}$ and $\{\om_2\}$ are arbitrary
simple, positively oriented loops in $\C$ with the following properties:

$\bullet$ Each of the contours surrounds the finite interval $[0,\sqrt\xi]$ and
leaves outside the semi--infinite interval $[1/\sqrt\xi,+\infty)\subset\R$.

$\bullet$ On the direct product of the contours, $\om_1\om_2\ne1$, so that the
denominator $\om_1\om_2-1$ in \eqref{f5.1} does not vanish.

The simplest contours satisfying these conditions are the circles centered at
0, with radii slightly greater than 1. However, to pass to the limit as
$\xi\to1$, we will deform these contours to a more sophisticated form, as
explained below.

To make the integrand meaningful, we have to specify the branches of the power
functions entering \eqref{f5.1}, and this is done in the following way. For the
terms $(1-\sqrt{\xi}\om)^\al$ (where $\om$ stands for $\om_1$ or $\om_2$ and
$\al$ equals $z'+x-\frac12$ or $z+y-\frac12$), we use the fact that $\{\om\}$
is contained in the simply connected region $\C\setminus[1/\sqrt\xi,+\infty)$
and specify the branch by setting $\arg(1-\sqrt{\xi}\om)=0$ for real negative
values of $\om$. Likewise, the terms $\om\mapsto (1-\sqrt{\xi}/\om)^\al$ are
well defined in the simply connected region
$(\C\cup\{\infty\})\setminus[0,\sqrt\xi]$, with the convention that
$\arg(1-\sqrt{\xi}/\om)=0$ for real $\om$ greater than $\sqrt\xi$.

Finally, note that the $\Ga$--factors in the numerator are not singular because
their arguments are not integers. Indeed, parameters $z$ and $z'$ are forbidden
to take integral values while $x-\frac12$ and $y-\frac12$ are integers. As for
the $\Ga$--factors in the denominator, their product is strictly positive,
again by virtue of the basic conditions on the parameters. Thus, we may and do
assume that the square root extracted from this product is positive, too.

To perform the limit transition as $\xi\to1$ we will need a special contour
$C(R,r,\xi)$ in the complex $\om$--plane. This contour depends on the
parameters $R>0$, $r>0$, and $\xi$, and looks as follows (see Fig.~1):

\medskip

\begin{center}
\scalebox{1.0}{\includegraphics{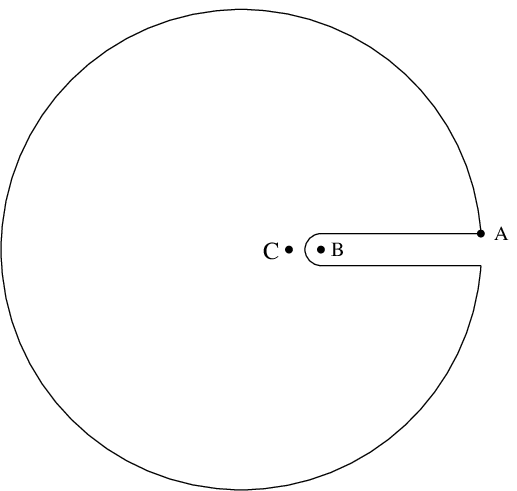}}

Figure 1. The contour $C(R,r,\xi)$: $A=Re^{i\theta}$, $B=1/\sqrt\xi$,
$C=\sqrt\xi$.
\end{center}

\medskip

\noindent Here we assume that the parameter $r>0$ is small enough while the
parameter $R>0$ is big enough. The integration path starts at the point $A=R
e^{i\theta}$ with small $\theta>0$ such that $\Im\om=R\sin\theta=r$, first goes
along the circle $|\omega|=R$ in the positive direction till the point
$Re^{-i\theta}$, then goes in parallel to the real line until the point
$1/\sqrt{\xi}-ir$, further goes to the point $1/\sqrt{\xi}+ir$ along the left
semicircle $|\om-1/\sqrt{\xi}|=r$ (so that 0 and $\sqrt\xi$ are left on the
left), and finally returns to the initial point $Re^{i\theta}$ in parallel to
the real line. \footnote{In \cite{BO3} the definition of the contour was
slightly different: there we assumed that before and after going around 0, the
path goes exactly on the positive real axis. In the context of the present
section such a definition works equally well but it is not suitable for the
integral representation appearing in Section \ref{6}. This is why we have to
slightly modify the definition of \cite{BO3}.} Note that $|\om|>1$ along the
whole contour provided that $r$ is so small that $1/\sqrt\xi-r>1$. This
condition also implies that $\sqrt\xi$ lies inside the contour, as required.

Next, consider the following contour in the complex $u$--plane (see Fig.~2).
Here $\rho>0$ is the parameter, the integration path starts at infinity, goes
towards 0 in the right half--plane, along the line $\Im u=-\rho$, then turns
around 0 in the negative direction along the semicircle $|u|=\rho$, and finally
returns to infinity in the right half--plane, along the line $\Im u=\rho$. Let
us denote this contour as $[+\infty-i\rho,0-,+\infty+i\rho]$.

\medskip

\begin{center}
\scalebox{1.0}{\includegraphics{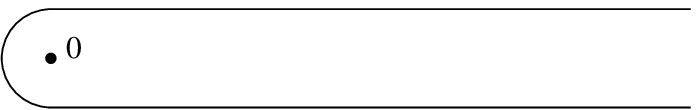}}

Figure 2. The contour $[+\infty-i\rho,0-,+\infty+i\rho]$.
\end{center}

\medskip

\begin{proposition}\label{5.1} Assume  $x,y\in\Z'$ are fixed. There exists the
limit
$$
\lim_{\xi\to1}\unK_{z,z',\xi}(x,y)=\unK_{z,z'}(x,y)
$$
with
\begin{multline}\label{f5.2}
\unK_{z,z'}(x,y)=
 \frac{\Ga(-z'-x+\frac12)\Ga(-z-y+\frac12)}
 {\bigl(\Ga(-z-x+\frac12)\Ga(-z'-x+\frac12)
 \Ga(-z-y+\frac12)\Ga(-z'-y+\frac12)\bigr)^{\frac12}}\\
\times\frac1{(2\pi i)^2}\oint\limits_{\{u_1\}} \oint\limits_{\{u_2\}}
\frac{(-u_1)^{z'+x-\frac12}(1+u_1)^{-z-x-\frac12}
(-u_2)^{z+y-\frac12}(1+u_2)^{-z'-y-\frac12}\,du_1 du_2}{u_1+u_2+1},
\end{multline}
where both contours $\{u_{1,2}\}$ are of the form
$[+\infty-i\rho,0-,+\infty+i\rho]$ as defined above, with $\rho<\frac12$.
\end{proposition}

\medskip
\noindent{\it Comments\/}. To give a sense to the function $(-u)^\al$ with
$\al\in\C$ we cut the complex $u$--plane along $[0,+\infty)$ and agree that the
argument of $-u$ equals 0 when $u$ intersects the negative real axis
$(0,-\infty)$. This is equivalent to say that the argument of $-u$ equals
$+\pi$ just below the cut and $-\pi$ just above the cut. Thus, one could remove
the minus sign from $(-u_1)^{z'+x-\frac12}$ and $(-u_2)^{z+y-\frac12}$ and put
instead in front of the integral the extra factor $e^{i\pi(z'+z+x+y-1)}$, with
the understanding that the argument of $u$ is equal to 0 just below the cut and
to $-2\pi$ just above it.

We also assume that the branch of $(1+u)^C$ (where $u$ is $u_1$ or $u_2$ and
$C$ equals $-z-x-\frac12$ or $z+y-\frac12$) is defined with the understanding
that the argument of $1+u$ is in $(-\pi/2,\pi/2)$.
\medskip

\begin{proof} The existence of the limit kernel was first established in
\cite{BO2}. In that paper we worked with the integrable form of the kernels
$\unK_{z,z',\xi}(x,y)$ and $\unK_{z,z'}(x,y)$, as written down above in
Subsection \ref{1-4}. In the present form, the claim of the proposition is a
particular case of a more general result obtained in \cite[\S9]{BO3}. I will
reproduce, with minor variations and with more details, the argument of
\cite{BO3} because all its steps will be employed in the sequel.
\footnote{Incidentally I will also correct minor inaccuracies in
\cite[\S9]{BO3}.}

{\it Step\/} 1. Let us check that the double integral in \eqref{f5.2} is
absolutely convergent:
\begin{equation}\label{f5.3}
\oint\limits_{\{u_1\}} \oint\limits_{\{u_2\}}
\left|\frac{(-u_1)^{z'+x-\frac12}(1+u_1)^{-z-x-\frac12}
(-u_2)^{z+y-\frac12}(1+u_2)^{-z'-y-\frac12}\,du_1
du_2}{u_1+u_2+1}\right|<+\infty
\end{equation}

First of all, the restriction $\rho<\frac12$ guarantees that the denominator
$u_1+u_2+1$ remains separated from 0 as $u_1$ and $u_2$ range over the
contours.

Set $\mu=\Re(z'-z)$. For the principal series $\mu=0$, and for the
complementary series $-1<\mu<1$.

Note that the modulus of the integrand is bounded from above, so that we have
to check the convergence only in the case when at least one of the variables
goes to infinity.

Fix a constant $C>0$ large enough. In the region where $|u_1|\le C$ and
$|u_2|\le C$, as was already pointed out, there is no problem of convergence.

Assume $|u_2|\le C$ while $|u_1|\ge C$. Then we may exclude $u_2$. That is, we
replace the quantity $\left|(-u_2)^{z+y-\frac12}(1+u_2)^{-z'-y-\frac12}\right|$
by an appropriate constant and also use the bound
$$
\frac1{|u_1+u_2+1|}\le\const\frac1{|u_1|}.
$$
This allows one to discard integration over $u_2$. Further, for large $|u_1|$,
on our contour, the arguments of $|u_1|$ and $|1+u_1|$ are small, which makes
it possible to replace both $u_1$ and $1+u_1$ by the real variable $u=\Re u_1$.
Then we are lead to the one--dimensional integral
$$
\int_C^{+\infty}u^{\mu-2}du
$$
whose convergence is obvious because $\mu<1$. Interchanging $u_1\leftrightarrow
u_2$ gives the same effect, due to the symmetry $z\leftrightarrow z'$.

Assume now that both variables are large, $|u_1|\ge C$ and $|u_2|\ge C$. Then
the same argument as above leads to the real integral
$$
\iint\limits_{u_1\ge C, u_2\ge
C}\frac{u_1^{\mu-1}u_2^{-\mu-1}}{u_1+u_2}du_1du_2.
$$
To handle it we use the bound
$$
\frac1{u_1+u_2}\le\frac{1}{u_1^\nu u_2^{1-\nu}}\,,
$$
which holds for any $\nu\in(0,1)$. Let us choose $\nu=\frac12+\frac12\mu$; then
$1-\nu=\frac12-\frac12\mu$. Since $\mu\in(-1,1)$, the requirement $\nu\in(0,1)$
is satisfied. This bound reduces our double integral to the product of two
simple integrals of the form
$$
\int_{u\ge C}u^{\pm\frac12\mu-\frac32}du,
$$
which are convergent.

{\it Step\/} 2. Let us turn to the kernel \eqref{f5.1}. Consider the contour
$\{\om\}=C(R, r,\xi)$ where $r=r(\xi)=(1-\xi)\rho$ with $\rho<\frac12$, as
above, and $R>0$ is large enough and fixed. It is readily verified that
$1/\sqrt\xi-r(\xi)>1$. As mentioned above, this inequality guarantees that
$|\om|>1$ on the whole contour, so that  the both contours in \eqref{f5.1} can
be deformed to the form $C(R,(1-\xi)\rho,\xi)$ without changing the value of
the double integral.

Let us split each contour on two parts, the big arc on the circle $|\om|=R$,
which we will denote as $C^-(R, (1-\xi)\rho,\xi)$, and the rest (inside the
circle), denoted as $C^+(R, (1-\xi)\rho,\xi)$.

The $\Ga$--factors in \eqref{f5.1} and in \eqref{f5.2} are the same, so that we
may ignore them. On the contrary, the prefactor $1-\xi$ in \eqref{f5.1} will
play the key role.

Our plan for the remainder of the proof is as follows: First, we show that when
we restrict the double integral in \eqref{f5.1} (together with the prefactor
$1-\xi$) on the product of two copies of $C^+(R, (1-\xi)\rho,\xi)$ and pass to
the limit as $\xi\to1$, we get the integral in \eqref{f5.2}. Next, we check
that the contribution from the rest of the double integral is asymptotically
negligible.

{\it Step\/} 3. Here we assume that both $\om_1$ and $\om_2$ range over the
contour $C^+(R, (1-\xi)\rho,\xi)$. Make the change of variables $\om_1\to u_1$,
$\om_2\to u_2$ according to the relation
\begin{equation}\label{f5.4}
\om=\om_\xi(u)=\frac1{\sqrt\xi}+(1-\xi)u.
\end{equation}
After this transformation, the contour $C^+(R, (1-\xi)\rho,\xi)$ turns into a
truncation of the contour $[+\infty-i\rho,0-,+\infty+i\rho]$ (we have to impose
the constraint $\Re u\le (1-\xi)^{-1}\wt R$, where $\wt R=R-(1/\sqrt\xi)\approx
R-1$). As $\xi$ goes to 1, the threshold of the truncation shifts to the right,
and in the limit we get the whole contour $[+\infty-i\rho,0-,+\infty+i\rho]$.

Substituting
\begin{gather*}
1-\sqrt\xi\om=-(1-\xi)\sqrt\xi\cdot u\\
\om-\sqrt\xi=(1-\xi)\left(u+\frac1{\sqrt\xi}\right)\\
\om_1\om_2-1=\frac{1-\xi}\xi\,\left(\sqrt\xi(u_1+u_2)+1+\xi(1-\xi)u_1u_2\right)\\
d\om_1 d\om_2=(1-\xi)^2du_1du_2
\end{gather*}
into the integrand of \eqref{f5.1}, we see that all the terms consisting of
various powers of $1-\xi$, including the prefactor $1-\xi$, cancel out.  The
integrand that we get can be written as the product of the integrand of
\eqref{f5.2} with the following expression depending on $\xi$:
\begin{equation}\label{f5.5}
\begin{gathered}
\xi^{\frac12(z+z'+x+y-1)}\left(\frac{1+u_1}{\frac1{\sqrt\xi}+u_1}\right)^{z+x+\frac12}
\left(\frac{1+u_2}{\frac1{\sqrt\xi}+u_2}\right)^{z'+y+\frac12}\\
\times\frac{u_1+u_2+1}{\sqrt\xi(u_1+u_2)+1+\xi(1-\xi)u_1u_2}\,
(\om_\xi(u_1))^z(\om_\xi(u_2))^{z'}
\end{gathered}
\end{equation}

As $\xi$ goes to 1, this expression converges to 1 pointwise (note that
$\om_\xi(u)\to1$). On the other hand, as $u_1$ and $u_2$ range over the contour
$[+\infty-i\rho, 0-,+\infty+i\rho]$, the modulus of this expression remains
bounded, uniformly on $\xi$. Indeed, this is obvious for the first four
factors. As for the last term, $|(\om_\xi(u_1))^z(\om_\xi(u_2))^{z'}|$, it is
bounded because $|\om_\xi(u)|\le R$ and the argument of $\om_\xi(u)$ is close
to 0 for $\xi$ close to 1 (recall that $\om_\xi(u)$ is close to $[1,+\infty)$).

Consequently, by virtue of step 1, our integral converges absolutely and
uniformly on $\xi$, so that we may pass to the limit under the sign of the
integral, which results in the desired integral \eqref{f5.2}.

{\it Step\/} 4. On this last step, we will check that the contribution from the
remaining parts of the contours, together with the prefactor $1-\xi$, is
asymptotically negligible. To do this, let us evaluate the modulus of the
integrand in \eqref{f5.1}.

The crucial observation is that the factor $\om_1\om_2-1$ in the denominator
remains separated from 0. Indeed, when at least one of the variables ranges
over the big arc $C^-(R,(1-\xi)\rho,\xi)$ (which is just our case), we have
$|\om_1\om_2|\ge R>1$. Consequently, we may ignore this factor, and then our
double integral factorizes into the product of two one--dimensional integrals:
one is
$$
\oint\limits_{\{\om_1\}} \left|\left(1-\sqrt{\xi}\om_1\right)^{z'+x-\tfrac12}
\left(1-\frac{\sqrt{\xi}}{\om_1}\right)^{-z-x-\tfrac12}
\om_1^{-x-\frac12}\,d\om_1\right|
$$
and the other has the same form, only  $x$ is replaced by $y$ and $z$ is
interchanged with $z'$:
$$
\oint\limits_{\{\om_2\}} \left|\left(1-\sqrt{\xi}\om_2\right)^{z+y-\tfrac12}
\left(1-\frac{\sqrt{\xi}}{\om_2}\right)^{-z'-y-\tfrac12}
\om_2^{-y-\frac12}\,d\om_2\right|.
$$

If one of the variables $\om_1, \om_2$ ranges over the big arc
$C^-(R,(1-\xi)\rho,\xi)$, then the integrand of the corresponding integral is
bounded uniformly on $\xi$, so that this integral remains bounded. Thus, the
case when both $\om_1$ and $\om_2$ range over the big arc is trivial: the
prefactor $1-\xi$ forces the whole expression to go to 0.

Examine now the case when one of the variables (say, $\om_1$) ranges over
$C^+(R,(1-\xi)\rho,\xi)$, while the other variable (hence, $\om_2$) ranges over
the big arc. Then we are left with the first integral together with the
prefactor $1-\xi$.

Let us make the same change of a variable as above: $\om_1=\om_\xi(u)$.  Then
the integral reduces to
$$
(1-\xi)^\mu(\sqrt\xi)^{z'+x-\frac12}\oint\limits_{\{u\}}
\left|(-u)^{z'+x-\frac12}\left(u+\frac1{\sqrt\xi}\right)^{-z-x-\frac12}
(\om_\xi(u))^zdu\right|.
$$
The quantity $|(\om_\xi(u))^z|$ is bounded uniformly on $\xi$ and hence may be
ignored. The factor $(\sqrt\xi)^{z'+x-\frac12}$ is inessential, too. Further,
using the same argument as on step 1, we reduce our expression  to the real
integral
$$
(1-\xi)^\mu\int_{C}^{C_1/(1-\xi)}u^{\mu-1}du
$$
where the upper limit arises due to the fact that $\Re(\om_\xi(u))\le R$.
Recall that $\mu=\Re(z'-z)\in(-1,1)$.

If $\mu\in(-1,0)$ then the integral is uniformly convergent, so that the whole
expression grows as  $(1-\xi)^\mu$.

If $\mu=0$ then the integral grows as $\log\left((1-\xi)^{-1}\right)$, and so
is the growth of the whole expression.

If $\mu\in(0,1)$ then the integral grows as $(1-\xi)^{-\mu}$ and the whole
expression remains bounded.

In all these cases the small prefactor $(1-\xi)$ in \eqref{f5.1} dominates and
makes the result asymptotically negligible.
\end{proof}

\begin{proposition}\label{5.2} The matrix
$$
\left[(m+\tfrac12)^{-1/2}
(K_{z,z'})_{++}(m+\tfrac12,n+\tfrac12)(n+\tfrac12)^{-1/2}\right]_{m,n=0,1,2,\dots}
$$
is of trace class.
\end{proposition}

\begin{proof} We have to prove that
$$
\sum_{m=0}^\infty
(m+\tfrac12)^{-1}(K_{z,z'})_{++}(m+\tfrac12,m+\tfrac12)<+\infty.
$$

Since the $++$ blocks of $K_{z,z'}$ and $\unK_{z,z'}$ are the same, we may use
formula \eqref{f5.2}. On the diagonal $x=y=m+\frac12$, the ratio in
\eqref{f5.2} formed by the $\Ga$--factors equals 1, so that we may omit them.
Therefore,
\begin{multline*} (K_{z,z'})_{++}(m+\tfrac12,m+\tfrac12)\\=\frac1{(2\pi
i)^2}\oint\limits_{\{u_1\}}
\oint\limits_{\{u_2\}}\frac{(-u_1)^{z'+m}(1+u_1)^{-z-m-1}
(-u_2)^{z+m}(1+u_2)^{-z'-m-1}\,du_1 du_2}{u_1+u_2+1}
\end{multline*}
where the contours are the same as in \eqref{f5.2}.

We have
\begin{multline*}
 (K_{z,z'})_{++}(m+\tfrac12,m+\tfrac12)\\
 \le \frac1{4\pi^2}\oint\limits_{\{u_1\}} \oint\limits_{\{u_2\}}
 \left| \frac{(-u_1)^{z'+m}(1+u_1)^{-z-m-1}
(-u_2)^{z+m}(1+u_2)^{-z'-m-1}\,du_1 du_2}{u_1+u_2+1}\right|\\
=\frac1{4\pi^2}\oint\limits_{\{u_1\}}
\oint\limits_{\{u_2\}}\left|\frac{u_1u_2}{(1+u_1)(1+u_2)}\right|^m\cdot\left|
\frac{(-u_1)^{z'}(1+u_1)^{-z-1}(-u_2)^{z}(1+u_2)^{-z'-1}du_1
du_2}{u_1+u_2+1}\right|\,.
\end{multline*}

Introduce the factor $(m+\tfrac12)^{-1}$ inside the integral and sum over
$m=0,1,\dots$. Observe that
$$
\gathered \sum_{m=0}^\infty
(m+\tfrac12)^{-1}\left|\frac{u_1u_2}{(1+u_1)(1+u_2)}\right|^m\\
\le 2+2\sum_{m=1}^\infty m^{-1}\left|\frac{u_1u_2}{(1+u_1)(1+u_2)}\right|^m\\
=2+2\log\left(\dfrac1{1-\left|\dfrac{u_1u_2}{(1+u_1)(1+u_2)}\right|}\right)=:F(u_1,u_2).
\endgathered
$$
Thus, we have to prove that
$$
\oint\limits_{\{u_1\}} \oint\limits_{\{u_2\}}F(u_1,u_2)\cdot\left|
\frac{(-u_1)^{z'}(1+u_1)^{-z-1}(-u_2)^{z}(1+u_2)^{-z'-1}du_1
du_2}{u_1+u_2+1}\right|<+\infty.
$$

Without the factor $F(u_1,u_2)$, this integral coincides with the integral
\eqref{f5.3} with $x=y=\frac12$, whose convergence has already been verified
(see step 1 in the proof of Proposition \ref{5.1}). Let us show that the extra
factor $F(u_1,u_2)$ does not add very much. Indeed, it grows only as both $u_1$
and $u_2$ go to infinity, so that we may assume that both $u_1$ and $u_2$ are
far from the origin. If $u$ is a point on one of the contours, far from the
origin, then writing $u=|u|e^{i\theta}$ we have $|u|$ large and $\theta$ small.
Then
$$
\left|\frac{u}{1+u}\right|^2
=\frac{|u|^2}{|u|^2+2|u|\cos\theta+1}=1-2|u|^{-1}\cos\theta+O(|u|^{-2})
$$
and consequently
$$
\left|\frac{u}{1+u}\right|=1-|u|^{-1}\cos\theta+O(|u|^{-2})
$$
with $\cos\theta\to1$ as $u\to\infty$. This allows one to estimate the growth
of the logarithm in $F(u_1,u_2)$. Omitting unessential details we get that it
behaves roughly as
$$
\log\left(\frac1{|u_1|^{-1}+|u_2|^{-1}}\right)\le\const |u_1|^\de|u_2|^\de,
$$
where $\de>0$ can be chosen arbitrarily small.

Using this bound and examining the argument of step 1 in Proposition \ref{5.1}
we see that the same arguments work equally well with the extra factor
$F(u_1,u_2)$.
\end{proof}

In the sequel we use the notation
$$
\epsi=1-\xi.
$$

\begin{lemma}\label{5.3} Fix an arbitrary $c\in(0,\frac12)$. If $R$ is large enough
and $\rho<\frac12$ then for all sufficiently small $\epsi$ the following
estimate holds
$$
\left|\frac{1-\sqrt{\xi}\om}{\om-\sqrt\xi}\right|<1-c\epsi, \qquad \om\in
C(R,\epsi\rho,\xi).
$$
\end{lemma}

\begin{proof} Consider the transform $\om\mapsto
\om'=\frac{1-\sqrt{\xi}\om}{\om-\sqrt\xi}$. Its inverse has the form
$\om=\frac{1+\sqrt{\xi}\om'}{\om'+\sqrt\xi}$ and sends the interior part of the
circle $|\om'|=1-c\epsi$ to the exterior part $S^{\operatorname{ext}}$ of a
circle $S$. Thus, we have to check that $C(R,\epsi\rho,\xi)\subset
S^{\operatorname{ext}}$.

The circle $S$ is symmetric relative the real axis and intersects it at the
points $\om_\mp$ corresponding to $\om'_\mp=\mp(1-c\epsi)$. Let us prove that
both $\om_-$ and $\om_+$ are on the left of $\frac1{\sqrt\xi}-\epsi\rho$, the
leftmost point of the contour $C^+(R,\epsi\rho,\xi)$.

We have
$$
\om_-=\frac{1-\sqrt{\xi}(1-c\epsi)}{-(1-c\epsi)+\sqrt\xi}
=\frac{1-(1-\tfrac12\epsi+O(\epsi^2))(1-c\epsi)}{-(1-c\epsi)+(1-\tfrac12\epsi+O(\epsi^2))}
=\frac{c+\tfrac12}{c-\tfrac12}+O(\epsi)
$$
and
$$
\om_+=\frac{1+\sqrt{\xi}(1-c\epsi)}{(1-c\epsi)+\sqrt\xi}
=1+\frac{(1-\sqrt{\xi})c\epsi}{(1-c\epsi)+\sqrt\xi} =1+O(\epsi^2).
$$

Since $c<\frac12$, $\om_-$ lies on the left of 0. As for $\om_+$, it lies on
the left of
$$
\frac1{\sqrt\xi}-\epsi\rho\approx 1+(\tfrac12-\rho)\epsi,
$$
because $\rho<\frac12$ by the assumption.

This shows that the interior part of the contour, $C^+(R,\epsi\rho,\xi)$ lies
in $S^{\operatorname{ext}}$. The same also holds for the exterior part,
$C^-(R,\epsi\rho,\xi)$, because $R$ can be made arbitrarily large. Indeed, as
seen from the above expressions for the points $\om_-$ and $\om_+$, they do not
run to infinity as $\xi\to1$, so that if $R$ is chosen large enough, the circle
of radius $R$ will enclose the circle $S$ for any $\xi$.
\end{proof}

\begin{proposition}\label{5.4}
\begin{multline*}
\lim_{\xi\to1}\left(\sum_{m\ge0}(m+\tfrac12)^{-1}
(K_{z,z',\xi})_{++}(m+\tfrac12,m+\tfrac12)\right)\\
=\sum_{m\ge0}(m+\tfrac12)^{-1} (K_{z,z'})_{++}(m+\tfrac12,m+\tfrac12).
\end{multline*}
\end{proposition}

\begin{proof} We will argue as in the proof of Proposition \ref{5.1}, steps 2--4. The
role of step 1 in that proof will be played by Proposition \ref{5.2}.

First of all, observe that on the diagonal $x=y$, the expression formed by the
$\Ga$--factors in front of the integrals in \eqref{f5.1} and \eqref{f5.2}
equals 1. Consequently, we may ignore these factors.

Let $F_{m;\xi}(\om_1,\om_2)$ denote the integrand in the integral \eqref{f5.1}
corresponding to $x=y=m+\frac12$. Likewise, let $F_m(u_1,u_2)$ denote the
integrand in the integral \eqref{f5.2} with $x=y=m+\frac12$.  We have to prove
that
\begin{multline}\label{f5.6}
\lim_{\xi\to1}\left((1-\xi)\sum_{m=0}(m+\tfrac12)^{-1}
\oint\limits_{\{\om_1\}}\oint\limits_{\{\om_2\}}F_{m;\xi}(\om_1,\om_2)d\om_1d\om_2\right)\\
=\sum_{m=0}(m+\tfrac12)^{-1}
\oint\limits_{\{u_1\}}\oint\limits_{\{u_2\}}F_m(u_1,u_2)du_1du_2\,.
\end{multline}

We already dispose of all the necessary information to conclude that
\eqref{f5.6} holds provided that we truncate the both contours in the
left--hand side to $C^+(R,\epsi\rho,\xi)$. Indeed, the ratio
$$
\frac{(1-\xi)F_{m;\xi}(\om_\xi(u_1),\om_\xi(u_2))}{F_m(u_1,u_2)}
$$
is the particular case of \eqref{f5.5} corresponding to $x=y=m+\frac12$. The
argument of step 3 in Proposition \ref{5.1} shows  that this expression goes to
1 pointwise. Moreover, its modulus is bounded uniformly on both $\xi$ and $m$:
To see this observe that
$$
\left|\frac{1+u}{\frac1{\sqrt\xi}+u}\right|\le 1, \qquad u=u_1,u_2,
$$
and recall that $|\om_\xi(u)|\le R_1$, $u=u_1,u_2$.

Together with the result of Proposition \ref{5.2} this implies the claim.

Now we have to check that the contribution of the remaining parts of the
contours to the left--hand side of \eqref{f5.6} is asymptotically negligible.
This can be done by slightly modifying the argument of step 4 in Proposition
\ref{5.1}.

Indeed, the integrand $F_{m,\xi}(\om_1,\om_2)$ can be written in the form
$$
F_{m;\xi}(\om_1,\om_2)=F_{0;\xi}(\om_1,\om_2)
\left(\frac{1-\sqrt{\xi}\om_1}{\om_1-\sqrt\xi}\right)^m
\left(\frac{1-\sqrt{\xi}\om_2}{\om_2-\sqrt\xi}\right)^m .
$$
Compare our task with the one we had on step 4 in Proposition \ref{5.1} (where
we set $x=y=\frac12$). The only difference is that now we have in the integral
the extra factor equal to
\begin{equation}\label{f5.7}
\sum_{m=0}^\infty(m+\tfrac12)^{-1}\left|\left(\frac{1-\sqrt{\xi}\om_1}{\om_1-\sqrt\xi}\right)
\left(\frac{1-\sqrt{\xi}\om_2}{\om_2-\sqrt\xi}\right)\right|^m.
\end{equation}
By Lemma \ref{5.3}, the quantity \eqref{f5.7} is majorated by
$$
\sum_{m=0}^\infty(m+\tfrac12)^{-1}(1-c\xi)^{2m},
$$
which grows as $\log(\epsi^{-1})$. Such an extra factor does not affect the
estimates on step 4 of Proposition \ref{5.1}.
\end{proof}

\section{Convergence of off--diagonal blocks in Hilbert--Schmidt norm}\label{6}

Here we prove the claims of Theorem \ref{4.3} for the off--diagonal blocks. As
in Section \ref{5}, we apply the symmetry relations of Proposition \ref{1.9} to
reduce the case of the $-+$ block to that of the $+-$ block, and due to the
relations of Subsection \ref{1-5}, we may freely switch from $K_{z,z',\xi}$ to
$\unK_{z,z',\xi}$.

However, to compute the Hilbert--Schmidt norm of the $+-$ block we cannot use
anymore the basic contour integral representation \eqref{f5.1} as we did in
Section \ref{5}, because this would lead to divergent series in the integrand.
The reason of this is that the variable $y$, which previously ranged over
$\Z'_+$, will now range over $\Z'_-$. It turns out that we still may employ
essentially the same estimates and arguments as in Section \ref{5} but
beforehand we have to change the integral representation \eqref{f5.1}.

Return to the initial version of the contours $\{\om_1\}$ and $\{\om_2\}$ in
\eqref{f5.1}, when they were circles slightly greater than the unit circle, and
make the change of a variable $\om_2\to\om_2^{-1}$. Then the former condition
$\om_1\om_2\ne1$ will turn into the requirement that the second contour must
lie inside the first contour. Further, we deform the contours so that they take
the form $C(R_1,\epsi\rho_1,\xi)$ and $C(R_2,\epsi\rho_2,\xi)$, respectively,
where $R_1>R_2$ and $\rho_1<\rho_2$: these inequalities just guarantee that the
contour $C(R_2,\epsi\rho_2,\xi)$ lies inside $C(R_1,\epsi\rho_1,\xi)$.

After the transform $\om_2\to\om_2^{-1}$ the formula \eqref{f5.1} turns into
\begin{equation}\label{f6.1}
\begin{aligned}
 \unK_{z,z',\xi}&(x;y)\\
&=\frac{\Gamma(-z'-x+\frac12)\Gamma(-z-y+\frac12)}
{\bigl(\Gamma(-z-x+\frac12)\Gamma(-z'-x+\frac12)
\Gamma(-z-y+\frac12)\Gamma(-z'-y+\frac12)\bigr)^{\frac12}} \\
&\times\frac{1-\xi}{(2\pi i)^2}\oint\limits_{\{\om_1\}}\oint\limits_{\{\om_2\}}
\left(1-\sqrt{\xi}\om_1\right)^{z'+x-\tfrac12}
\left(1-\frac{\sqrt{\xi}}{\om_1}\right)^{-z-x-\tfrac12}\\
&\times\left(1-\sqrt{\xi}\om_2\right)^{-z'-y-\tfrac12}
\left(1-\frac{\sqrt{\xi}}{\om_2}\right)^{z+y-\tfrac12} \frac{\om_1^{-x-\frac
12}\om_2^{y-\frac 12}}{\omega_1-\omega_2} \,{d\om_1}{d\om_2}.
\end{aligned}
\end{equation}

According to Proposition \ref{5.1}, for any fixed $x,y\in\Z'$, there exists a
limit
$$
\lim_{\xi\to1}\unK_{z,z',\xi}(x,y)=\unK_{z,z'}(x,y),
$$
for which we dispose of the double contour integral representation
\eqref{f5.2}. However, now we will need a different integral representation,
which is consistent with \eqref{f6.1}:

\begin{proposition}\label{6.1} The following formula holds
\begin{multline}\label{f6.2}
\unK_{z,z'}(x,y)=
 \frac{\Ga(-z'-x+\frac12)\Ga(-z-y+\frac12)}
 {\bigl(\Ga(-z-x+\frac12)\Ga(-z'-x+\frac12)
 \Ga(-z-y+\frac12)\Ga(-z'-y+\frac12)\bigr)^{\frac12}}\\
\times\frac1{(2\pi i)^2}\oint\limits_{\{u_1\}} \oint\limits_{\{u_2\}}
\frac{(-u_1)^{z'+x-\frac12}(1+u_1)^{-z-x-\frac12}
(-u_2)^{-z'-y-\frac12}(1+u_2)^{z+y-\frac12}\,du_1 du_2}{u_1-u_2}.
\end{multline}
where the contours have the form
$$
\{u_1\}=[+\infty-i\rho_1,0-,+\infty+i\rho_1],\qquad
\{u_2\}=[+\infty-i\rho_2,0-,+\infty+i\rho_2]
$$
and $\rho_1<\rho_2$.
\end{proposition}

\begin{proof} The $\Ga$--factors in \eqref{f6.1} and \eqref{f6.2} are the same,
so that it suffices to prove that  the integral in \eqref{f6.1} together with
the prefactor $(1-\xi)$ converges to the integral in \eqref{f6.2}. We will
follow the scheme of the proof of Proposition \ref{5.1}.

{\it Step\/} 1. Let us check that that the integral \eqref{f6.2} is absolutely
convergent. Arguing as on step 1 of Proposition \ref{5.1} we reduce this claim
to finiteness of the real integral
\begin{equation}\label{f6.3}
\iint\limits_{\substack{u_1,u_2\ge C\\ |u_1-u_2|\ge \rho}}
\frac{u_1^{\mu-1}u_2^{-\mu-1}} {|u_1-u_2|}du_1 du_2,
\end{equation}
where $\mu=\Re(z'-z)\in(-1,1)$, $\rho=\rho_2-\rho_1>0$, and $C$ is a positive
constant.

By symmetry we may assume $u_1\ge u_2$. Making the change of variables
$u_1=u+a$, $u_2=u$, where $u\ge C$, $a\ge \rho$, we get the integral
\begin{gather*}
\int_{a\ge \rho}a^{-1}da\int_{u\ge C} (u+a)^{\mu-1}u^{-\mu-1}du\\
=\int_{a\ge \rho}a^{-2}da\int_{u\ge C/a} (u+1)^{\mu-1}u^{-\mu-1}du.
\end{gather*}
Consider the interior integral in the second line: For large $u$, the integrand
decays as $u^{-2}$, which is integrable near infinity, while for $u$ near 0,
the integrand behaves as $u^{-\mu-1}$. It follows that if $\mu<0$ then the
integral over $u$ converges uniformly on $a$; if $\mu=0$ then it grows as $\log
a$ for large $a$; and if $\mu>0$ then it grows as $a^{\mu}$. Consequently, the
double integral can be estimated by one of the three convergent integrals
$$
\int_{a\ge \rho}a^{-2}da, \qquad \int_{a\ge \rho}a^{-2}\log a\, da,  \qquad
\int_{a\ge \rho}a^{\mu-2}da \quad (0<\mu<1).
$$

{\it Step\/} 2. Relying on the result of step 1, we can verify the required
convergence of the integrals provided that we restrict the integration in
\eqref{f6.1} to interior parts of the contours. This is done exactly as on step
3 of Proposition \ref{5.1}. A minor simplification is that under the change of
a variable \eqref{f5.4}, the quantity $\om_1-\om_2$ is transformed simpler than
$\om_1\om_2-1$.

{\it Step\/} 3. Finally, we have to prove that the contributions from the
remaining parts of the contours is asymptotically negligible due to the
prefactor $1-\xi$. Here we argue exactly as on step 4 of Proposition \ref{5.1},
with the only exception: In the case when $\om_1$ ranges over the interior part
of the contour, $C^+(R_1,\epsi\rho_1,\xi)$, while $\om_2$ ranges over the
exterior part, $C^-(R_2,\epsi\rho_2,\xi)$, we cannot automatically discard the
denominator $\om_1-\om_2$. The reason is that in this special case, the two
parts of the contours come close at the distance of order $\epsi$, which
happens near the point $R_2$.

This difficulty can be resolved in the following way. Dissect
$C^+(R_1,\epsi\rho_1,\xi)$ into two parts by a vertical line $\Re \om_1=\const$
so that the points on the left be separated from $R_2$ while the points on the
right be separated from $1/\sqrt\xi$. Now we have two cases:

When $\om_1$ ranges over the left part, we may discard the denominator
$\om_1-\om_2$ and argue as on step 4 of Proposition \ref{5.1}.

When $\om_1$ ranges over the right part, the argument is different: Observe
that the modulus of the whole integrand in \eqref{f6.1}, except the denominator
$\om_1-\om_2$, is bounded uniformly on $\xi$, while the quantity
$|\om_1-\om_2|^{-1}$ is integrable, so that the prefactor $1-\xi$ makes the
contribution negligible. Checking the integrability of $|\om_1-\om_2|^{-1}$ is
easy: Indeed, here it is even unessential that $C^-(R_2,\epsi\rho_2,\xi)$ does
not contain a small arc around the point $R_2$. What is important is that the
singularity $|\om_1-\om_2|^{-1}$ arising near the point $R_2$ is of the same
kind as the singularity $(a^2+b^2)^{-1/2}$ in the real $(a,b)$--plane near the
origin.
\end{proof}

Although formulas \eqref{f6.1} and \eqref{f6.2} are valid for any $x,y\in\Z'$,
we will deal exclusively with $x\in\Z'_+$ and $y\in\Z'_-$. Below we set
$$
x=m+\tfrac12, \quad y=-n-\tfrac12, \qquad m,n\in\Z_+
$$
and use the notation
\begin{gather*}
(K_{z,z',\xi})_{+-}(m,n)=(-1)^n\unK_{z,z',\xi}(m+\tfrac12,-n-\tfrac12)\\
(K_{z,z'})_{+-}(m,n)=(-1)^n\unK_{z,z'}(m+\tfrac12,-n-\tfrac12).
\end{gather*}
Here the factor $(-1)^n$ comes from the factor $\epsi(y)$ in \eqref{f1.2}.

The next result is similar to Proposition \ref{5.2}:

\begin{proposition}\label{6.2} The matrix
$$
\left[(m+\tfrac12)^{-1/2}(K_{z,z'})_{+-}
(m,n)(n+\tfrac12)^{-1/2}\right]_{m,n\in\Z'_+}
$$
belongs to the Hilbert--Schmidt class.
\end{proposition}

\begin{proof}
Since all the matrix entries are real (Remark \ref{1.7}), the claim means that
the series
\begin{equation}\label{f6.4}
\sum_{m,n=0}^\infty(m+\tfrac12)^{-1}(n+\tfrac12)^{-1}((K_{z,z'})_{+-} (m,n))^2
\end{equation}
is finite.

To get the squared matrix entry we multiply out two copies of the double
integral representation \eqref{f6.2}, where in the second copy we swap $z$ and
$z'$. This operation does not change the kernel, as it follows from a series
expansion for the kernel (see \cite[(3.3)]{BO4}) and the fact that the
functions entering this expansion depend symmetrically on $z$ and $z'$. On the
other hand, after this operation the $\Ga$--factors in front of the integral
will cancel out (this trick is borrowed from \cite{BO3} and \cite{BO4}, see,
e.g., the proof of Proposition 2.3 in \cite{BO4}). The resulting expression for
the sum \eqref{f6.4} can be written in the form
\begin{multline}\label{f6.5}
\sum_{m,n\in\Z'_+}(m+\tfrac12)^{-1}(n+\tfrac12)^{-1}\\
\times\frac1{16\pi^4}\oint\limits_{\{u_1\}}\oint\limits_{\{u_2\}}
\oint\limits_{\{u'_1\}}\oint\limits_{\{u'_2\}}
F_{mn}(u_1,u_2)F'_{mn}(u'_1,u'_2) du_1d u_2 du'_1 du'_2,
\end{multline}
where
\begin{gather*}
\{u_1\}=\{u'_1\}=[+\infty-i\rho_1,0-,+\infty+i\rho_1], \\
\{u_2\}=\{u'_2\}=[+\infty-i\rho_2,0-,+\infty+i\rho_2]
\end{gather*}
and
\begin{gather*}
F_{mn}(u_1,u_2)= \frac{(-u_1)^{z'+m}(1+u_1)^{-z-m-1}
(-u_2)^{-z'+n}(1+u_2)^{z-n-1}}{u_1-u_2}\,,\\
F'_{mn}(u'_1,u'_2)= \frac{(-u'_1)^{z+m}(1+u'_1)^{-z'-m-1}
(-u'_2)^{-z+n}(1+u'_2)^{z'-n-1}}{u'_1-u'_2}
\end{gather*}
(in the second line $z$ and $z'$ are interchanged).

To show that the sum \eqref{f6.5} is finite we replace the integrand by its
modulus and then interchange summation and integration. Then we get the
integral
\begin{multline}\label{f6.6}
\frac{1}{16\pi^4}\oint\limits_{\{u_1\}}\oint\limits_{\{u_2\}}
\oint\limits_{\{u'_1\}}\oint\limits_{\{u'_2\}}
\sum_{m,n\in\Z'_+}(m+\tfrac12)^{-1}(n+\tfrac12)^{-1}\\
\times\left|F_{mn}(u_1,u_2)F'_{mn}(u'_1,u'_2) du_1d u_2 du'_1 du'_2\right|.
\end{multline}
It suffices to check that it is finite.

We have
\begin{multline*}
\sum_{m,n=0}^\infty(m+\tfrac12)^{-1}(n+\tfrac12)^{-1}\left|F_{mn}(u_1,u_2)F'_{mn}(u'_1,u'_2)\right|
=\left|F_{00}(u_1,u_2)F'_{00}(u'_1,u'_2)\right|\\
\times\sum_{m,n=0}^\infty(m+\tfrac12)^{-1}(n+\tfrac12)^{-1}\left|\frac{u_1u'_1}{u_1+u'_1}\right|^m
\left|\frac{u_2u'_2}{u_2+u'_2}\right|^n.
\end{multline*}
The same argument as in the proof of Proposition \ref{5.2} shows that for the
latter sum there exists the upper bound of the form
$$
\sum_{m,n=0}^\infty(m+\tfrac12)^{-1}(n+\tfrac12)^{-1}\left|\frac{u_1u'_1}{u_1+u'_1}\right|^m
\left|\frac{u_2u'_2}{u_2+u'_2}\right|^n \le\const|u_1|^\de |u'_1|^\de |u_2|^\de
|u'_2|^\de,
$$
where $\de>0$ can be chosen as small as is needed.

Substituting this estimate into the 4--fold integral \eqref{f6.6} leads to its
splitting into the product of two double integrals, one of which is
\begin{equation}\label{f6.7}
\oint\limits_{\{u_1\}}\oint\limits_{\{u_2\}} |u_1|^\de
|u_2|^\de\left|\frac{(-u_1)^{z'}(1+u_1)^{-z-1}
(-u_2)^{-z'}(1+u_2)^{z-1}}{u_1-u_2} du_1d u_2 \right|
\end{equation}
and the other has the similar form, with $z$ and $z'$ interchanged. Therefore,
it suffices to prove the finiteness of the integral \eqref{f6.7}.

Arguing as on step 1 of Proposition \ref{5.1} we reduce this integral to the
real integral
\begin{equation}\label{f6.8}
\iint\limits_{\substack{u_1,u_2\ge C\\ |u_1-u_2|\ge \rho}}
\frac{u_1^{\mu+\de-1}u_2^{-\mu+\de-1}} {|u_1-u_2|}du_1 du_2.
\end{equation}
This integral only slightly differs from the integral \eqref{f6.3} examined on
step 1 of Proposition \ref{6.1}, and the same argument as in Proposition
\ref{6.1} shows that \eqref{f6.8} is finite provided that $\de$ is chosen small
enough.
\end{proof}

\begin{proposition}\label{6.3} As $\xi$ goes to $1$, the matrices
$$
\left[(m+\tfrac12)^{-1/2}(K_{z,z',\xi})_{+-}
(m+\tfrac12,-n-\tfrac12)(n+\tfrac12)^{-1/2}\right]_{m,n\in\Z'_+}
$$
converge to the matrix
$$
\left[(m+\tfrac12)^{-1/2}(K_{z,z'})_{+-}
(m+\tfrac12,-n-\tfrac12)(n+\tfrac12)^{-1/2}\right]_{m,n\in\Z'_+}
$$
in the topology of the Hilbert--Schmidt norm.
\end{proposition}

\begin{proof} We know that the convergence takes place in the weak operator
topology, that is, for the matrix entries (this follows from Proposition
\ref{5.1}). Consequently, it suffices to prove the convergence of the squared
Hilbert--Schmidt norms:
\begin{multline*}
\sum_{m,n=0}^\infty(m+\tfrac12)^{-1}(n+\tfrac12)^{-1}((K_{z,z',\xi})_{+-}
(m,n))^2\\ \to
\sum_{m,n=0}^\infty(m+\tfrac12)^{-1}(n+\tfrac12)^{-1}((K_{z,z'})_{+-} (m,n))^2
\end{multline*}
(recall that our matrices are real).

We will follow the arguments of Propositions \ref{5.4}, \ref{6.1}, and
\ref{5.1}.

Write the left--hand side in the form similar to \eqref{f6.5}. Namely, let
$F_{mn;\xi}(\om_1,\om_2)$ denote the integrand in \eqref{f6.1}, where we
substitute $x=m+\frac12$ and $y=-n-\frac12$:
\begin{multline*}
 F_{mn;\xi}(\om_1,\om_2)=\left(1-\sqrt{\xi}\om_1\right)^{z'+m}
\left(1-\frac{\sqrt{\xi}}{\om_1}\right)^{-z-m-1}\\
\times\left(1-\sqrt{\xi}\om_2\right)^{-z'+n}
\left(1-\frac{\sqrt{\xi}}{\om_2}\right)^{z-n-1}
\frac{\om_1^{-m-1}\om_2^{-n-1}}{\omega_1-\omega_2}\,,
\end{multline*}
and let $F'_{mn;\xi}(\om'_1,\om'_2)$ be the similar quantity with $z$ and $z'$
interchanged:
\begin{multline*}
F'_{mn;\xi}(\om'_1,\om'_2)=\left(1-\sqrt{\xi}\om_1\right)^{z+m}
\left(1-\frac{\sqrt{\xi}}{\om_1}\right)^{-z'-m-1}\\
\times\left(1-\sqrt{\xi}\om_2\right)^{-z+n}
\left(1-\frac{\sqrt{\xi}}{\om_2}\right)^{z'-n-1}
\frac{(\om'_1)^{-m-1}(\om'_2)^{-n-1}}{\omega'_1-\omega'_2}\,.
\end{multline*}
As above, swapping $z$ and $z'$ kills the $\Ga$--factors. In this notation, the
series in question takes the form
\begin{multline}\label{f6.9}
\frac{1}{16\pi^4}\sum_{m,n\in\Z'_+}(m+\tfrac12)^{-1}(n+\tfrac12)^{-1}\\
\times\oint\limits_{\{\om_1\}}\oint\limits_{\{\om_2\}}
\oint\limits_{\{\om'_1\}}\oint\limits_{\{\om'_2\}} (1-\xi)^2
F_{mn;\xi}(\om_1,\om_2)F'_{mn;\xi}(\om'_1,\om'_2) d\om_1d\om_2d\om'_1d\om'_2\,,
\end{multline}
where
$$
\{\om_1\}=\{\om'_1\}=C(R_1,\epsi\rho_1,\xi), \qquad
\{\om_2\}=\{\om'_2\}=C(R_2,\epsi\rho_2,\xi),
$$
and $R_1>R_2$, $r_1<r_2$, $\epsi=1-\xi$, as before.

First, we truncate all the contours in \eqref{f6.9} keeping only their interior
parts $C^+(R_i,\epsi\rho_i,\xi)$, and then make the change of variables
according to the rule \eqref{f5.4}, as we did before. The prefactor $(1-\xi)^2$
disappears, we compare the resulting integrand with that in \eqref{f6.5}, and
check that their ratio has uniformly bounded modulus and converges pointwise to
1. This is done exactly as in step 3 of the proof of Proposition \ref{5.4}. By
virtue of Proposition \ref{6.3}, we get the desired convergence provided that
the both contours in \eqref{f6.9} are truncated.

Next, we check that the contribution from the remaining parts of the contours
in \eqref{f6.9} is asymptotically negligible due to the prefactor $(1-\xi)^2$.
Again, the proof goes as in the situation of Proposition \ref{5.4}. We replace
the integrand by its modulus, interchange summation and integration, and
evaluate the double sum over $m$ and $n$ using Lemma \ref{5.3}. This produces a
factor growing like $(\log(\epsi^{-1}))^2$, which we add to our small prefactor
$(1-\xi)^2$. Due to this bound, the 4--fold integral reduces to the product of
two double integrals,
$$
\oint\limits_{\{\om_1\}}\oint\limits_{\{\om_2\}}
|F_{00;\xi}(\om_1,\om_2)d\om_1d\om_2| \qquad \text{\rm and} \qquad
\oint\limits_{\{\om'_1\}}\oint\limits_{\{\om'_2\}}
|F'_{00;\xi}(\om'_1,\om'_2)d\om'_1d\om'_2|,
$$
where each of the 4 variables ranges over $C^\pm(R_i,\epsi\rho_i,\xi)$, and for
at least one of them the corresponding superscript has to be ``$-$''.

These integrals were already examined in the proof of Proposition \ref{6.1}.
For each of the integrals, there are two possible cases: Either (a) both
variables range over interior parts of the contours or (b) one of the variables
ranges over the interior part while the other variable ranges over the exterior
part. We know that in case (a) the integral grows as $(1-\xi)^{-1}$, while in
case (b) the growth is suppressed by the small factor $(1-\xi)$, even if one
adds the extra growing factor $(\log(\epsi^{-1}))^2$. Since case (b) occurs for
at least one of the double integrals, we see that the small prefactor
$(1-\xi)^2$ suffices to make the whole contribution negligible.

This completes the proof.
\end{proof}

\end{document}